\documentclass{amsart}
\usepackage{color}
\usepackage{graphicx}
\usepackage{hyperref}
\usepackage{amssymb}
\usepackage{amsfonts}
\usepackage{euscript}
\usepackage[all]{xy}
\usepackage{bbm}

\numberwithin{equation}{section}
\numberwithin{figure}{section}
\renewcommand{\subsection}{\hspace{-\parindent}\refstepcounter{subsection}{\bf \arabic{section}\alph{subsection}. }}

\newenvironment{nouppercase}{%
  \renewcommand{\uppercasenonmath}[1]{}}{}

\theoremstyle{plain}
\newtheorem{thm}{Theorem}[section]
\newtheorem{theorem}[thm]{Theorem}

\newtheorem{assumption}[thm]{Assumption}
\newtheorem{definition}[thm]{Definition}

\newtheorem{remark}[thm]{Remark}

\newtheorem{proposition}[thm]{Proposition}
\newtheorem{example}[thm]{Example}
\newtheorem{non-example}[thm]{Non-example}

\newtheorem{lemma}[thm]{Lemma}

\newtheorem*{claim*}{Claim} 
\newtheorem*{lemma*}{Lemma}
\newtheorem*{theorem*}{Theorem}
\newtheorem*{conjecture*}{Conjecture}

\newcommand{\bC}{{\mathbb C}}

\newcommand{\bF}{{\mathbb F}}

\newcommand{\bQ}{{\mathbb Q}}
\newcommand{\bR}{{\mathbb R}}

\newcommand{\bZ}{{\mathbb Z}}

\newcommand{\scrC}{\EuScript C}

\newcommand{\scrM}{\EuScript M}

\newcommand{\iso}{\cong}
\newcommand{\htp}{\simeq}
\newcommand{\smooth}{C^\infty}

\title[Quantum Steenrod]{\Large\larger\rm Covariant constancy of quantum Steenrod operations}
\author{Paul Seidel, Nicholas Wilkins}

\date{\today. This is an unfinished draft.}

\begin{document}
\begin{nouppercase}
\maketitle
\end{nouppercase}

\begin{abstract}
We prove a relationship between quantum Steenrod operations and the quantum connection. In particular there are operations extending the quantum Steenrod power operations that, when viewed as endomorphisms of equivariant quantum cohomology, are covariantly constant. We demonstrate how this property is used in computations of examples.
\end{abstract}

\section{Introduction}
Quantum Steenrod operations, originally introduced by Fukaya \cite{fukaya93b}, have recently appeared in a variety of contexts: their properties have been explored in \cite{wilkins18} (which also contains the first nontrivial computations); they can be used to study arithmetic aspects of mirror symmetry \cite{seidel19}; and in Hamiltonian dynamics, they are relevant for the existence of pseudo-rotations \cite{shelukhin19,cineli-ginzburg-gurel19,shelukhin19b}. Nevertheless, computing quantum Steenrod operations remains a challenging problem in all but the simplest cases. Using methods similar to \cite{wilkins18}, this paper establishes a relation between quantum Steenrod operations and the quantum connection. As a consequence, the contribution of rational curves of low degree (very roughly speaking, of degree $<p$ if one is interested in quantum Steenrod operations with $\bF_p$-coefficients) can be computed using only ordinary Steenrod operations and Gromov-Witten invariants. This is consonant with other indications that the geometrically most interesting part of quantum Steenrod operations may come from $p$-fold covered curves. Even though our method does not reach that part, it yields interesting results in many examples (some are carried out here, and there are more in \cite{seidel19}).

\subsection{}
Throughout this paper, $M$ is a closed symplectic manifold which is weakly monotone \cite{hofer-salamon95} (in \cite[Definition 6.4.1]{mcduff-salamon}, this is called semi-positive). Fix an arbitrary coefficient field $\bF$. The associated Novikov ring $\Lambda$ is the ring of series
\begin{equation} \label{eq:novikov}
\gamma = \textstyle\sum_A c_A q^A,
\end{equation}
where the exponents are $A \in H_2^{\mathit{sphere}}(M;\bZ) = \mathrm{im}(\pi_2(M) \rightarrow H_2(M;\bZ))$ such that either $A = 0$ or $\int_A \omega_M > 0$; and among those $A$ such that $\int_A \omega_M$ is bounded by a given constant, only finitely many $c_A$ may be nonzero. We think of this as a graded ring, where $|q^A| = 2 c_1(A)$ (the notation being that $c_1(A)$ is the pairing between $c_1(M)$ and $A$). Write $I_{\mathit{max}} \subset \Lambda$ for the ideal generated by $q^A$ for nonzero $A$, so that $\Lambda/I_{\mathit{max}} = \bF$. 

For each $a \in H^2(M;\bZ)$ there is an $\bF$-linear differentiation operation $\partial_a: \Lambda \rightarrow \Lambda$, 
\begin{equation} \label{eq:derivative}
\partial_a  q^A = (a \cdot A)\, q^A.
\end{equation}
Write $I_{\mathit{diff}} \subset I_{\mathit{max}}$ for the ideal generated by $q^A$, where $A \neq 0$ lies in the kernel of the map $H_2^{\mathit{sphere}}(M;\bZ) \hookrightarrow H_2(M;\bZ) \twoheadrightarrow \mathit{Hom}(H^2(M;\bZ),\bF)$. In other words, the generators are precisely those nontrivial monomials whose derivatives \eqref{eq:derivative} are zero. (If $\bF$ is of characteristic zero and $H_*^{\mathit{sphere}}(M;\bZ)$ is torsion-free, then $I_{\mathit{diff}} = 0$; but that's not the case we'll be interested in.)

\begin{remark}
Clearly, $\partial_a$ only depends on $a \otimes 1 \in H^2(M;\bZ) \otimes \bF$. One could define such operations for all elements in $H^2(M;\bF)$, and prove a version of our results in that context. We have refrained from doing so, since it adds a technical wrinkle (having to represent classes in $H^2(M;\bF)$ geometrically) without giving any striking additional applications.
\end{remark}

\subsection{}
We will exclusively consider genus zero Gromov-Witten invariants. The three-pointed Gromov-Witten invariant in a class $A \in H_2^{\mathit{sphere}}(M;\bZ)$ can be written as a bilinear operation 
\begin{equation}
\begin{aligned}
& \ast_A: H^*(M;\bF)^{\otimes 2} \longrightarrow H^{*-2c_1(A)}(M;\bF), \\
& \int_M (c_1 \ast_A c_2)\, c_3 = \langle c_1,c_2,c_3 \rangle_A.
\end{aligned}
\end{equation}
One extends this to $H^*(M;\Lambda)$, and then packages all the $\ast_A$ into the small quantum product
\begin{equation}
\gamma_1 \ast \gamma_2 = \sum_A (\gamma_1 \ast_A \gamma_2)\, q^A.
\end{equation}
Let $t$ be another formal variable, of degree $2$. The quantum connection on $H^*(M;\Lambda)[[t]]$ consists of the operations
\begin{equation} \label{eq:nabla}
\nabla_a \gamma = t\partial_a \gamma + a \ast \gamma,
\end{equation}
where $\ast$ has been extended $t$-linearly. By the divisor axiom in Gromov-Witten theory, we have that for any $a_1,a_2 \in H^2(M;\bZ)$ and $c_1,c_2 \in H^*(M;\bF)$,
\begin{equation}
 (a_1 \cdot A) \int_M (a_2 \ast_A c_1)\, c_2 = \langle a_1, a_2, c_1, c_2 \rangle_{A} = (a_2 \cdot A) \int_M (a_1 \ast_A c_1)\, c_2. 
\end{equation}
This implies that the operations \eqref{eq:nabla} for different $a$ commute: the connection is flat. 

We will consider endomorphisms $\Sigma$ of $H^*(M;\Lambda)[[t]]$ which are $\Lambda[[t]]$-linear and covariantly constant, which means that they satisfy
\begin{equation} \label{eq:covariantly-constant}
\nabla_a \Sigma - \Sigma \nabla_a = 0.
\end{equation}
This is a system of linear first order differential equations. By looking at the equations for each $q^A$ coefficient of $\Sigma$, one sees that:

\begin{lemma} \label{th:up-to-order-p}
For covariantly constant endomorphisms, the constant term determines the behaviour modulo $I_{\mathit{diff}}$. More formally, if $\Sigma$ satisfies \eqref{eq:covariantly-constant}, then we have
\begin{equation}
\Sigma \in \mathit{End}(H^*(M;\bF)) \otimes I_{\mathit{max}}[[t]] 
 \;\; \Longrightarrow \;\; 
\Sigma \in \mathit{End}(H^*(M;\bF)) \otimes I_{\mathit{diff}}[[t]].
\end{equation}
\end{lemma}

\subsection{}
From now on, we restrict to coefficient fields $\bF = \bF_p$, for a prime $p$. Our arguments involve $(\bZ/p)$-equivariant cohomology with $\bF_p$-coefficients. For a point, that is
\begin{equation} \label{eq:equivariant-ring}
H^*_{\bZ/p}(\mathit{point};\bF_p) = H^*(B\bZ/p;\bF_p) = \bF_p[[t,\theta]], \;\; |t| = 2,\, |\theta| = 1.
\end{equation}
The notation requires some explanation. For $p = 2$, we have $\theta^2 = t$, so $\bF_2[[t,\theta]]$ is actually a ring of power series in a single variable $\theta$. For $p > 2$, we have $t \theta = \theta t$ and $\theta^2 = 0$, so that $\bF_p[[t,\theta]]$ is a ring of power series in two supercommuting variables. 

For any $A \in H_2^{\mathit{sphere}}(M;\bZ)$ and any class $b \in H^*(M;\bF_p)$, one can use $(\bZ/p)$-equivariant Gromov-Witten theory to define an operation
\begin{equation} \label{eq:q-sigma-a}
Q\Sigma_{b,A}: H^*(M;\bF_p) \longrightarrow (H^*(M;\bF_p)[[t,\theta]])^{*+p|b|-2c_1(A)}.
\end{equation}
For the trivial class $A = 0$, this is a form of the classical Steenrod operation $\mathit{St}$, more precisely
\begin{equation} \label{eq:classical-xi}
Q\Sigma_{b,0}(c) = \mathit{St}(b) c.
\end{equation}

\begin{remark}
Our notational and sign conventions follow \cite{seidel19} (except that we suppress the prime $p$), which differ from the classical conventions for Steenrod operations. In particular, for $p>2$,
\begin{equation} \label{eq:trivial-steenrod}
\mathit{St}(b) = (-1)^{\frac{|b|(|b|-1)}{2} \frac{p-1}{2}} \big({\textstyle\frac{p-1}{2}!}\big)^{|b|} t^{\frac{p-1}{2} |b|}\, b + \cdots,
\end{equation}
where $\cdots$ is the part involving cohomology classes of degree $>|b|$. For $|b|$ even, this simplifies to
\begin{equation} \label{eq:trivial-steenrod-2}
\mathit{St}(b) = (-1)^{\frac{|b|}{2}} t^{\frac{p-1}{2}|b|} b + \cdots
\end{equation}
At the other extreme, setting $t = \theta = 0$ in $\mathit{St}(b)$ still yields the $p$-fold (cup) power $b^p$. The Cartan relation says that
\begin{equation} \label{eq:cartan}
\mathit{St}(\tilde{b}) \,\mathit{St}(b) = (-1)^{|b|\,|\tilde{b}|\,\frac{p(p-1)}{2}} \mathit{St}(\tilde{b} b).
\end{equation}
Note that many coefficients of $\mathit{St}(b)$ vanish, because this operation comes from the cohomology of the symmetric group. Concretely, if $|b|$ is even, all the potentially nonzero terms in $\mathit{St}(b)$ are of the form $t^{k(p-1)}$ or $t^{k(p-1)-1}\theta$; and if $|b|$ is odd, of the form $t^{(k+1/2)(p-1)}$ or $t^{(k+1/2)(p-1)-1}\theta$. That is no longer true for quantum operations.
\end{remark}

As usual, one adds up \eqref{eq:q-sigma-a} over all $A$ with weights $q^A$. The outcome is denoted by
\begin{equation} \label{eq:sigma}
Q\Sigma_b: H^*(M;\bF_p) \longrightarrow (H^*(M;\Lambda)[[t,\theta]])^{*+p|b|}.
\end{equation}
The non-equivariant ($t = \theta = 0$) part is the $p$-fold quantum product with $b$:
\begin{equation} \label{eq:non-equivariant}
Q\Sigma_b(c) = \overbrace{b \ast \cdots \ast b}^p \ast c + (\text{\it terms involving $t,\theta$}).
\end{equation}
The case $b = 1$ is trivial:
\begin{equation} \label{eq:steenrod-zero}
Q\Sigma_1 = \mathit{id}.
\end{equation}
The relation with the more standard formulation of the quantum Steenrod operation is that
\begin{equation} \label{eq:steenrod-sigma}
\mathit{QSt}(b) = Q\Sigma_b(1).
\end{equation}

It is convenient to formally extend \eqref{eq:sigma}. First, turn it into an endomorphism of $H^*(M;\Lambda)[[t,\theta]]$, linearly in the variables $q^A$ and $(t,\theta)$ (with appropriate Koszul signs). 
%
Next, extend the $b$-variable to $\beta \in H^*(M;\Lambda)$, by setting
\begin{equation} \label{eq:b-variable-extension}
\textstyle Q\Sigma_{\beta} = \sum_A q^{pA} \,Q\Sigma_{b_A} \quad \text{for $\beta = \sum_A b_A q^A$.}
\end{equation}
Then, the composition of these operations is described by
\begin{equation} \label{eq:compose-sigma}
Q\Sigma_{\tilde{b}} \circ Q\Sigma_{b} = (-1)^{|b|\,|\tilde{b}|\,\frac{p(p-1)}{2}} Q\Sigma_{\tilde{b}\, \ast\, b}.
\end{equation}
Note that for $b = 1$, \eqref{eq:non-equivariant} implies that $Q\Sigma_1$ is an automorphism of $H^*(M;\Lambda)[[t,\theta]]$, and \eqref{eq:compose-sigma} that it is idempotent. Hence, it must be the identity, so those two properties imply \eqref{eq:steenrod-zero}.

%
%

\subsection{}
The quantum connection can be extended to $H^*(M;\Lambda)[[t,\theta]]$ by making it $\theta$-linear. Our main result is:
%

\begin{theorem} \label{th:covariantly-constant}
For any $b \in H^*(M;\bF_p)$, the operation $Q\Sigma_b$ is a covariantly constant endomorphism (of degree $p|b|$), meaning that it satisfies \eqref{eq:covariantly-constant}.
\end{theorem}

Lemma \ref{th:up-to-order-p} still applies (the presence of the additional $\theta$-variable makes no difference). Hence, the classical part \eqref{eq:classical-xi}, together with the quantum connection, determine $Q\Sigma_b$ modulo $I_{\mathit{diff}}$. 

\begin{remark} \label{th:fundamental-solution}
Covariant constancy also means that $Q\Sigma_b$ is related to the fundamental solution of the quantum differential equation (see e.g.\ \cite{pandharipande96}). To explain this, let's temporarily switch coefficients to $\bQ$, and write $\tilde{\Lambda}$ for the associated Novikov ring. The fundamental solution is a trivialization of the quantum connection,
\begin{equation} \label{eq:fundamental-solution}
\nabla \tilde\Psi = 0,
\end{equation}
whose constant (in the $q$ variables) term is the identity endomorphism. $\tilde\Psi$ is multivalued (has $\log(q^A)$ terms), and is also a series in $t^{-1}$. It is uniquely determined by those conditions, and one can write down an explicit formula in terms of Gromov-Witten invariants with gravitational descendants. Given $\beta \in H^*(M;\bZ)$, write 
\begin{equation}
\tilde\Xi_\beta(\gamma) = \tilde\Psi(\beta \, \tilde\Psi^{-1}(\gamma)).
\end{equation}
By construction, this is a covariantly constant endomorphism, whose constant term is cup-product with $\beta$. It is single-valued; more precisely,
\begin{equation} \label{eq:xi-a}
\tilde\Xi_\beta \in \mathit{End}(H^*(M;\tilde\Lambda))[[t^{-1}]].
\end{equation}
For simplicity, suppose that $H^*(M;\bZ)$ is torsion-free. One can look at the denominators in $\tilde\Xi_\beta$, order by order in the covariant constancy equation. The upshot is that factors of $1/p$ appear for the first time in terms $q^A$, $A \in pH_2^{\mathit{sphere}}(M;\bZ)$. As a consequence, $\tilde\Xi_\beta$ has a well-defined partial reduction mod $p$, which we denote by
\begin{equation} \label{eq:xi-endomorphism}
\Xi_\beta \in \mathit{End}(H^*(M;\Lambda/I_{\mathit{diff}}))[[t^{-1}]],
\end{equation}
and which only depends on $\beta \in H^*(M;\bF_p)$. Let's extend \eqref{eq:xi-endomorphism} linearly to $\beta \in H^*(M;\bF_p)[t,\theta]$, in which case $\Xi_\beta$ can have both positive and negative powers of $t$. The case we are interested in is $\beta = \mathit{St}(b)$. Because of the uniqueness property from Lemma \ref{th:up-to-order-p}, we then have
\begin{equation}
\Xi_{\mathit{St}(b)} = Q\Sigma_b \;\;\text{ modulo $I_{\mathit{diff}}$.}
\end{equation}
\end{remark}

\begin{example} \label{th:s2}
Consider $M = S^2$, with the standard basis $\{1,h\}$ of cohomology. Take $p>2$ (the case $p = 2$ is straightforward, but requires slightly different notation). Using Theorem \ref{th:covariantly-constant}, one can compute that $Q\Sigma_h = -t^{p-1}\Sigma$, where
\begin{equation} \label{eq:explicit-xi-2}
\Sigma = \begin{pmatrix} \sigma_{11} & \sigma_{12} \\ \sigma_{21} & \sigma_{22} \end{pmatrix}, \;\;
\left\{
\begin{aligned}
& \textstyle \sigma_{11} = -\sum_{k=1}^{(p-1)/2} \frac{(2k-1)!}{(k!)^2(k-1)!^2} q^k t^{1-2k}, \\
& \textstyle \sigma_{12} = -\sum_{k=2}^{(p+1)/2} \frac{(2k-2)!}{(k-2)! \, (k-1)!^2 \, k!} q^k t^{2-2k}, \\
& \textstyle \sigma_{21} = \sum_{k=0}^{(p-1)/2} \frac{(2k)!}{(k!)^4}  q^k t^{-2k}, \\
& \textstyle \sigma_{22} = -\sigma_{11}.
\end{aligned}
\right.
\end{equation}
In particular, 
\begin{equation}
\label{eq:qst-sphere}
\mathit{QSt}(h) = -t^{p-1}\sigma_{11}\, 1 - t^{p-1}\sigma_{21}\, h.
\end{equation}
Note that after multiplying with $t^{p-1}$, all the powers of $t$ in \eqref{eq:explicit-xi-2} become nonnegative. More precisely,
\begin{equation}
-t^{p-1}\Sigma = \begin{pmatrix}  0 & q^{(p+1)/2} \\ q^{(p-1)/2} & 0 \end{pmatrix} + (\text{\it terms involving $t$}),
\end{equation}
in agreement with \eqref{eq:non-equivariant} and the fact that the $p$-th quantum power of $h$ is $q^{(p-1)/2}h$. This is proved in Section \ref{sec:computations}.
\end{example}

\begin{example} \label{th:cubic-surface-mod-2}
Let $M$ be a cubic surface in $\bC P^3$ (this is $\bC P^2$ blown up at $6$ points, with its monotone symplectic form). Take $p = 2$, and let $h \in H^4(M;\bF_2)$ be the Poincar{\'e} dual of a point. Then
\begin{equation}
\mathit{QSt}(h) = \mathit{St}(h) = t^2 h.
\end{equation}
This is interesting because of its implications for Hamiltonian dymanics: by the criterion from \cite{cineli-ginzburg-gurel19,shelukhin19b}, it means that $M$ cannot admit a pseudo-rotation. We refer to Section \ref{subsec:s2} for further discussion.
\end{example}

The proof of Theorem \ref{th:covariantly-constant} goes roughly as follows. We introduce another operation, depending on $a \in H^2(M;\bZ)$ as well as $b \in H^*(M;\bF_p)$,
\begin{equation} \label{eq:modified-sigma}
Q\Pi_{a,b}: H^*(M;\bF_p) \longrightarrow (H^*(M;\Lambda)[[t,\theta]])^{*+|a|-2+p|b|}.
\end{equation}
Geometrically, this is obtained from \eqref{eq:sigma} by equipping the underlying Riemann surface with an additional marked point, which can move around (we insert an incidence constraint dual to $a$ at that point). A localisation-type argument yields
\begin{equation} \label{eq:pi-relation}
t\, Q\Pi_{a,b}(c) = 
Q\Sigma_b(a \ast c) - a \ast Q\Sigma_b(c).
\end{equation}
We also have an analogue of the divisor equation:
\begin{equation} \label{eq:pi-divisor}
Q\Pi_{a,b}(c) = \partial_a Q\Sigma_b(c).
\end{equation}
Theorem \ref{th:covariantly-constant} follows immediately by combining \eqref{eq:pi-relation} and \eqref{eq:pi-divisor}. 

\begin{remark}
Even though we have no immediate need for it here, it is worth while noting that $Q\Pi_{a,b}$ can be defined more generally for $a \in H^*(M;\bF_p)$, and still satisfies \eqref{eq:pi-relation}, with suitable added Koszul signs (see Remark \ref{th:generalized-pi}).
\end{remark}

\begin{remark}
The argument above is closely related to the Cartan relation for quantum Steenrod squares. Namely, let's set $a = \mathit{QSt}(b_1)$, $b = b_2$, $c = 1$ in \eqref{eq:pi-relation}. Then, using \eqref{eq:compose-sigma} one sees that
\begin{equation}
\begin{aligned}
& t\, Q\Pi_{\mathit{QSt}(b_1),b_2}(1) = (-1)^{|b_1|\,|b_2|} Q\Sigma_{b_2}(\mathit{QSt}(b_1)) - \mathit{QSt}(b_1) \ast \mathit{QSt}(b_2) \\
& = (-1)^{|b_1|\,|b_2|\, (p(p-1)/2+1)} Q\Sigma_{b_2 \ast b_1}(1) - \mathit{QSt}(b_1) \ast \mathit{QSt}(b_2) \\
& = (-1)^{|b_1|\,|b_2|\, p(p-1)/2} \mathit{QSt}(b_1 \ast b_2) - \mathit{QSt}(b_1) \ast \mathit{QSt}(b_2). \\
\end{aligned}
\end{equation}
In view of that, it is not surprising that in applications, computations based on covariant constancy closely resemble those from \cite{wilkins18}, where the Cartan relation was the main tool.
\end{remark}

{\em Acknowledgments.} Both authors were partially supported by a Simons Investigator award from the Simons Foundation. Additional support for the first author was provided by the Simons Collaboration for Homological Mirror Symmetry, and by NSF grant DMS-1904997. The second author was additionally supported by a Heilbronn Research Fellowship.

\section{A bit of equivariant (co)homology\label{sec:basicparameterspaces}}
This section introduces some of the algebra and topology underlying our construction. Even though this is elementary, it is helpful as a guiding model for the later discussion.

\subsection{} 
Write
\begin{equation}
S^\infty = \{ w = (w_0,w_1,w_2,\dots) \in \bC^\infty \; : \; w_k = 0 \text{ for $k \gg 0$}, \; \|w\|^2 = |w_0|^2 + |w_1|^2 + \cdots = 1\}.
\end{equation}
Fix a prime $p$, and consider the $\bZ/p$-action on $S^\infty$ generated by
\begin{equation}
\tau(w_0,w_1,\dots) = (\zeta w_0, \zeta w_1,\dots), \;\; \zeta = e^{2\pi  i/p}.
\end{equation}
Take the following subsets:
\begin{align} \label{eq:d-delta-1}
& \Delta_{2k} = \{ w \in S^\infty \;:\; w_k \geq 0, \; w_{k+1} = w_{k+2} = \cdots = 0\}, \\
\label{eq:d-delta-2}
& \Delta_{2k+1} = \{ w \in S^\infty \;:\; e^{-i \theta} w_k \geq 0 \text{ for some $\theta \in [0,2\pi/p]$}, \; w_{k+1} = w_{k+2} = \cdots = 0\}.
\end{align}
\begin{figure}
\begin{centering}
\includegraphics{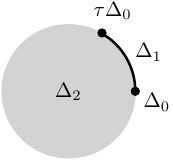}
\caption{The first cells from \eqref{eq:d-delta-1}, \eqref{eq:d-delta-2}.}
\label{fig:s-infinity-cells}
\end{centering}
\end{figure}%
Each of them is homeomorphic to a disc, of the dimension indicated by the subscript. More precisely, $\Delta_{2k}$ is a submanifold with boundary,
\begin{equation}
\partial \Delta_{2k} = \{w_k = w_{k+1} = \cdots = 0\} \iso S^{2k-1},
\end{equation}
and $\Delta_{2k+1}$ a submanifold with two boundary faces, whose intersection forms a corner stratum,
\begin{equation}
\partial \Delta_{2k+1} = \{w_k \geq 0,\, w_{k+1} = w_{k+2} = \cdots = 0\} \cup
\{e^{-2\pi i/p} w_k \geq 0, \, w_{k+1} = w_{k+2} = \cdots = 0\}.
\end{equation}
The subsets \eqref{eq:d-delta-1}, \eqref{eq:d-delta-2} and their images under the $\bZ/p$-action form an equivariant (and regular) cell decomposition of $S^\infty$. The tangent space of $\Delta_{2k}$ at the point where $w_k = 1$ (and where all the other coordinates are therefore zero)  can be identified with $\bC^k$ by projecting to the first $k$ coordinates; we use the resulting orientation. The tangent space of $\Delta_{2k+1}$ at the same point can be similarly identified with $\bC^k \times i\bR$; we use the orientation coming from the complex orientation of $\bC^k$, followed by the positive vertical orientation of $i\bR$. For those orientations, the differential in the cellular chain complex is
\begin{align} \label{eq:partial-delta-1}
& \partial \Delta_{2k} = \Delta_{2k-1} + \tau \Delta_{2k-1} + \cdots + \tau^{p-1} \Delta_{2k-1}, \\ \label{eq:partial-delta-2}
& \partial \Delta_{2k+1} = \tau \Delta_{2k} - \Delta_{2k}.
\end{align}
Here and below, the convention is to ignore terms with negative subscripts.

We adopt the quotient $S^\infty/(\bZ/p)$ as our model for the classifying space $B(\bZ/p)$. If we use $\bF_p$-coefficients, the $\Delta_i$ become cycles on the quotient, and their homology classes form a basis for $H^{\mathit{eq}}_*(\mathit{point};\bF_p) = H_*(S^\infty/(\bZ/p);\bF_p)$. (Moreover, from \eqref{eq:partial-delta-1} one sees that the Bockstein sends $\Delta_{2k}$ to $\Delta_{2k-1}$.) 

\subsection{\label{subsec:kunneth}}
Consider the diagonal embedding $\delta$ on $S^\infty/(\bZ/p)$, and the induced map
\begin{equation}
\delta_*: H_*(S^\infty/(\bZ/p);\bF_p) \longrightarrow (H_*(S^\infty/(\bZ/p);\bF_p))^{\otimes 2}.
\end{equation}

\begin{lemma}
In homology with $\bF_p$-coefficients, 
\begin{equation} \label{eq:kunneth-of-diagonal}
\delta_* \Delta_i = \begin{cases} \displaystyle \sum_{i_1+i_2 = i} \Delta_{i_1} \otimes \Delta_{i_2} & \text{if $i$ is odd or $p=2$,} \\[1em]
\displaystyle \sum_{\substack{i_1+i_2=i \\ \text{$i_k$ even}}} \Delta_{i_1} \otimes \Delta_{i_2} & \text{if $i$ is even and $p>2$.}
\end{cases}
\end{equation}
\end{lemma}

\begin{proof}
For $p = 2$, this is clear: from the relation between diagonal map and cup product, and the ring structure on the cohomology of $\bR P^\infty = S^\infty/(\bZ/2)$, we can see that $\delta_* \Delta_i$ must have nonzero components in all groups $H^{i_1} \otimes H^{i_2}$, and each of those is a copy of $\bF_2$.

For $p>2$, the same argument shows that exactly the terms in \eqref{eq:kunneth-of-diagonal} must occur, but possibly with some nonzero $\bF_p$-coefficients, which have to be determined by looking a little more carefully. Choose generators $\theta \in H^1(S^\infty/(\bZ/p);\bF_p)$ and $t \in H^2(S^\infty/(\bZ/p);\bF_p)$ so that 
\begin{equation}
\langle \theta, \Delta_1 \rangle = 1, \;\; \langle t, \Delta_2 \rangle = -1.
\end{equation}
Because $\Delta_2$ was defined using the complex orientation, this means that $t$ is the pullback of the (mod $p$) Chern class of the tautological line bundle $S^\infty \rightarrow \bC P^\infty$ under the quotient map $S^\infty/(\bZ/p) \rightarrow S^\infty/S^1 = \bC P^\infty$. Looking at the orientations of the higher-dimensional cells yields
\begin{equation}
\langle t^k\theta, \Delta_{2k+1} \rangle = \langle t^k, \Delta_{2k} \rangle = (-1)^k.
\end{equation}
For $k = k_1+k_2$, we have
\begin{align}
& \langle t^{k_1} \otimes t^{k_2}, \delta_* \Delta_{2k} \rangle =
\langle \delta^*(t^{k_1} \otimes t^{k_2}), \Delta_{2k}\rangle =
\langle t^k, \Delta_{2k} \rangle, \\
& \langle t^{k_1}\theta \otimes t^{k_2}, \delta_* \Delta_{2k+1} \rangle =
\langle \delta^*(t^{k_1}\theta \otimes t^{k_2}), \Delta_{2k+1} \rangle
= \langle t^k\theta, \Delta_{2k+1} \rangle,
\end{align}
and that implies that the coefficients in \eqref{eq:kunneth-of-diagonal} are all $1$, as desired.
\end{proof}

What does this mean on the cochain level? For each $k$, take a smooth triangulation of $S^{2k-1}/(\bZ/p)$. Pull that back (taking preimages of the simplices) to a triangulation of $\partial \Delta_{2k}$, and then extend that to a triangulation of $\Delta_{2k}$. The outcome is an explicit smooth singular chain in $S^\infty/(\bZ/p)$, denoted by $\tilde{\Delta}_{2k}$, which becomes a singular cycle when the coefficients are reduced modulo $p$, and which represents the homology class of $\Delta_{2k}$ in $H_*(S^\infty/(\bZ/p);\bF_p)$. A version of the same process produces corresponding singular chains $\tilde{\Delta}_{2k-1}$. With that in mind, let's look at the relations underlying \eqref{eq:kunneth-of-diagonal}:
\begin{equation}
\label{eq:kunneth-of-diagonal-2}
\delta \tilde\Delta_i \sim \begin{cases} \displaystyle \sum_{i_1+i_2 = i} \tilde\Delta_{i_1} \times \tilde\Delta_{i_2} & \text{if $i$ is odd or $p=2$,} \\[1em]
\displaystyle \sum_{\substack{i_1+i_2=i \\ \text{$i_k$ even}}} \tilde\Delta_{i_1} \times \tilde\Delta_{i_2} & \text{if $i$ is even and $p>2$.}
\end{cases}
\end{equation}
On the right hand side, one decomposes the products into simplices. After that, the relation means that there is a singular chain whose boundary (mod $p$) equals the difference between the two sides. That chain can again be chosen to be smooth. One could in principle try to spell all of this out using explicit chains, but that is not necessary for our purpose.

\subsection{\label{sec:subsec-basic-top-b}}
Consider the two-sphere $S = \bar{\bC} = \bC \cup \{\infty\}$, again with a $\bZ/p$-action 
$\sigma(v) = \zeta v$, and the subsets
\begin{align}
\label{eq:p0}
& P_0 = \{v = 0\}, \;\;
Q_0 = \{v = \infty\}, \\
& L_1 = \{v \geq 0\} \cup \{v = \infty\}, \\
\label{eq:b2}
& B_2 = \{e^{-i \theta} v \geq 0 \text{ for some $\theta \in [0,2\pi/p]$}\} \cup \{v = \infty\}.
\end{align}
We use the real orientation of $L_1$, and the complex orientation of $B_2$. Let's denote the associated cellular chain complex simply by $C_*(S)$. Its differential is
\begin{align}
\label{eq:diff-s2-1}
& \partial P_0 = \partial Q_0 = 0, \\
& \partial L_1 = Q_0 - P_0, \\
\label{eq:diff-s2-3}
& \partial B_2 = L_1 - \sigma L_1.
\end{align}
\begin{figure}
\begin{centering}
\includegraphics{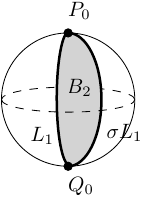}
\caption{The cells from \eqref{eq:p0}--\eqref{eq:b2}.}
\label{fig:s2-decomposition}
\end{centering}
\end{figure}%

Now look at $S^\infty \times_{\bZ/p} S$, which means identifying 
\begin{equation} \label{eq:id}
(w,\sigma v) \sim (\tau w, v). 
\end{equation}
This inherits a cell decomposition. The associated differential, which we denote by $\partial^{\mathit{eq}}$, is
\begin{align}
& 
\partial^{\mathit{eq}} (\Delta_{2k} \times P_0) =  0, 
\;\;
\partial^{\mathit{eq}} (\Delta_{2k+1} \times P_0) = 0, 
\\ & 
\partial^{\mathit{eq}} (\Delta_{2k} \times Q_0) =  0, 
\;\; 
\partial^{\mathit{eq}} (\Delta_{2k+1} \times Q_0) = 0, 
\\ \label{eq:cellular-product-3} &
\partial^{\mathit{eq}} (\Delta_{2k} \times \sigma^j L_1) =  \Delta_{2k} \times (Q_0 - P_0) 
+ \Delta_{2k-1} \times (L_1 + \sigma L_1 + \cdots + \sigma^{p-1} L_1),
\\ \label{eq:cellular-product-4} &
\partial^{\mathit{eq}} (\Delta_{2k+1} \times \sigma^j L_1) =  -\Delta_{2k+1} \times (Q_0-P_0)
+ \Delta_{2k} \times (\sigma^{j+1}L_1 - \sigma^j L_1),
\\ \label{eq:cellular-product-5} &
\partial^{\mathit{eq}} (\Delta_{2k} \times \sigma^j B_2) = - \Delta_{2k} \times (\sigma^{j+1} L_1  - \sigma^j L_1)
+ \Delta_{2k-1} \times (B_2 + \cdots + \sigma^{p-1}B_2),
\\ \label{eq:cellular-product-6} &
\partial^{\mathit{eq}} (\Delta_{2k+1} \times \sigma^j B_2) = \Delta_{2k+1} \times (\sigma^{j+1}L_1 -\sigma^jL_1)
+ \Delta_{2k} \times (\sigma^{j+1}B_2 - \sigma^j B_2).
\end{align}

\begin{lemma} \label{lemma:homologyrelations}
Take coefficients in $\bF_p$. In the cellular complex of $S^\infty \times_{\bZ/p} S$, the following homology relationships hold:
\begin{align} 
& \label{eq:even-relation}
\Delta_{2k} \times (Q_0 - P_0) \sim \Delta_{2k-2} \times (B_2 + \sigma B_2 + \cdots + \sigma^{p-1}B_2),
\\ & \label{eq:odd-relation}
\Delta_{2k+1} \times (Q_0 - P_0) \sim \Delta_{2k-1} \times (B_2 + \sigma B_2 + \cdots + \sigma^{p-1}B_2).
\end{align}
\end{lemma}

\begin{proof}
\eqref{eq:even-relation} is obtained by subtracting \eqref{eq:cellular-product-3} from the following, which comes from \eqref{eq:cellular-product-6}:
\begin{equation} \label{eq:1st}
\begin{aligned}
&
\partial^{\mathit{eq}} \big( \Delta_{2k+1} \times (\sigma B_2 + 2\sigma^2 B_2 + \cdots + (p-1)\sigma^{p-1}B_2) \big) 
\\ & \qquad \qquad 
= - \Delta_{2k+1} \times (L_1 + \cdots + \sigma^{p-1}L_1)
- \Delta_{2k} \times (B_2 + \cdots + \sigma^{p-1}B_2).
\end{aligned}
\end{equation}
The second relation \eqref{eq:odd-relation} is a combination of \eqref{eq:cellular-product-4}, \eqref{eq:cellular-product-5}.
\end{proof}

To fit this into the general framework of equivariant homology,
note that as an application of the localisation theorem, the map induced by inclusion of the fixed point set,
\begin{equation} \label{eq:include-fixed-set}
H_*^{\mathit{eq}}(\mathit{point};\bF_p) \otimes P_0 \oplus
H_*^{\mathit{eq}}(\mathit{point};\bF_p) \otimes Q_0
\longrightarrow H_*^{\mathit{eq}}(S;\bF_p) = H_*(S^\infty \times_{\bZ/p} S;\bF_p)
\end{equation}
must be an isomorphism in sufficiently high degrees. Using the computations above, one can see how that works out concretely: \eqref{eq:include-fixed-set} is surjective, and it fails to be injective only in degrees $0$ and $1$, where the kernel is generated by $\Delta_0 \otimes (Q_0-P_0)$ and $\Delta_1 \otimes (Q_0-P_0)$, respectively.

More generally, take any (homologically graded) chain complex, carrying a $(\bZ/p)$-action. Its equivariant homology is defined by taking the tensor product with the previously considered cellular complex of $S^\infty$, and then passing to coinvariants for the combined action in the same sense as in \eqref{eq:id}. The resulting equivariant differential is
\begin{align}
&
\partial^{\mathit{eq}}(\Delta_{2k} \otimes \xi) =  \Delta_{2k-1} \otimes (\xi + \sigma \xi + \cdots + \sigma^{p-1} \xi) + \Delta_{2k} \otimes \partial \xi,
\\ 
&
\partial^{\mathit{eq}}(\Delta_{2k+1} \otimes \xi) = -\Delta_{2k+1} \otimes \partial \xi + \Delta_{2k} \otimes (\sigma \xi - \xi).
\end{align}
Here, $\xi$ is an element of the original chain complex, and $\sigma$ is the automorphism which generates its $(\bZ/p)$-action. These formulae generalize the ones we've previously written down for $C_*(S)$.

\subsection{}
%
Dually to our previous construction, one can start with a cohomologically graded complex $C$ with a $(\bZ/p)$-action, and define an equivariant complex
\begin{equation} \label{eq:general-equivariant}
C_{\mathit{eq}} = C[[t,\theta]]
\end{equation}
where the formal variables are as in \eqref{eq:equivariant-ring}, with differential
\begin{align}
& 
d_{\mathit{eq}}(x t^k) = dx\, t^k + (-1)^{|x|} (\sigma x - x) t^k\theta, \\
&
d_{\mathit{eq}}(x t^k\theta) = dx\, t^k\theta + (-1)^{|x|} (x + \sigma x + \cdots + \sigma^{p-1} x) t^{k+1}.
\end{align}
Write $H^*_{\mathit{eq}}(C)= H^*(C_{\mathit{eq}})$ for the resulting cohomology. 

\begin{lemma} \label{th:trivial-sigma}
On $C_{\mathit{eq}}$, the operations $t$ and $\sigma t$ are homotopic.
\end{lemma}

\begin{proof}
The desired homotopy is $h(x t^k) = 0$, $h(x t^k\theta) = (-1)^{|x|} xt^{k+1}$. 
\end{proof}

From now on, we work with $\bF_p$-coefficients. In that case, the equivariant complex \eqref{eq:general-equivariant} carries a degree $1$ endomorphism $\tilde{\theta}$, which one can informally think of as a corrected version of multiplication with $\theta$ (acting on the left):
\begin{align} \label{eq:endomorphism}
& \tilde{\theta}(xt^k) = (-1)^{|x|} xt^k\theta, \\
\label{eq:endomorphism-2}
& \tilde{\theta}(xt^k\theta) = (-1)^{|x|} (\sigma x + 2\sigma^2 x + \cdots + (p-1)\sigma^{p-1} x)t^{k+1}.
\end{align}
The second part \eqref{eq:endomorphism-2} contains the kind of expression we've seen previously in \eqref{eq:1st}. It is helpful to keep in mind that modulo $p$,
\begin{align}
& \mathit{id} + \sigma + \sigma^2 + \cdots + \sigma^{p-1} = (\sigma-\mathit{id})^{p-1} = \sigma (\sigma-\mathit{id})^{p-1} = \cdots, \\
& \sigma + 2\sigma^2 + \cdots + (p-1)\sigma^{p-1} = -\sigma(\sigma-\mathit{id})^{p-2}.
\label{eq:sum-jj}
\end{align}
Using that, one sees that the map $\tilde{\theta}$ is a chain map (of degree $1$) with respect to $d_{\mathit{eq}}$: 
\begin{equation}
\begin{aligned}
& d_{\mathit{eq}} \tilde{\theta}(xt^k) = d_{\mathit{eq}}( (-1)^{|x|} xt^k\theta) 
= (-1)^{|x|} dx\, t^k\theta + (\mathit{id} + \sigma + \cdots)x \, t^{k+1}
\\ & \;\; = (-1)^{|x|} dx\, t^k\theta - (\sigma + 2\sigma^2 + \cdots)(\sigma-\mathit{id})x t^{k+1}
\\ & \;\; = \tilde{\theta} (-dx\, t^k - (-1)^{|x|} (\sigma-\mathit{id}) x \, t^k\theta)
= -\tilde{\theta} d_{\mathit{eq}}(xt^k),
\end{aligned}
\end{equation}
and similarly
\begin{equation}
\begin{aligned}
& d_{\mathit{eq}} \tilde{\theta}(x t^k\theta) = d_{\mathit{eq}}( (-1)^{|x|} 
(\sigma + 2\sigma^2 + \cdots) xt^{k+1})
\\ & \;\; = (-1)^{|x|} (\sigma + 2\sigma^2 + \cdots) dx\, t^{k+1} - (\mathit{id} + \sigma + \cdots) x t^{k+1} \theta 
\\ & \;\; = -\tilde{\theta}(dx \, t^k \theta + (-1)^{|x|} (\mathit{id} + \sigma + \cdots) xt^{k+1})
= -\tilde{\theta} d_{\mathit{eq}}(x t^k\theta).
\end{aligned}
\end{equation}


\begin{lemma}
Up to homotopy, $\tilde{\theta}^2$ is multiplication by $t$ if $p = 2$, and $0$ for $p>2$.
\end{lemma}

\begin{proof}
In terms of \eqref{eq:sum-jj}, $\tilde{\theta}^2$ is the action of $-\sigma(\sigma-\mathit{id})^{p-2}t$ on the equivariant complex. But the action of $(\sigma-\mathit{id})t$ is nullhomotopic by Lemma \ref{th:trivial-sigma}, and that implies the desired statement.
\end{proof}

A classical application of equivariant cohomology (basic to the definition of Steenrod operations) is to start with a general cochain complex $C$ (without any $(\bZ/p)$-action), and consider its $p$-fold tensor product $C^{\otimes p}$ with the action that cyclically permutes the tensor factors. The equivariant complex $(C^{\otimes \mathit p})_{\mathit{eq}}$ is a homotopy invariant of $C$. We recall the following:

\begin{lemma} \label{lemma:power-map-additive-t}
Taking a cocycle $x \in C$ to $x^{\otimes p} \in (C^{\otimes p})_{\mathit{eq}}$ yields a map
\begin{equation}
H^*(C) \longrightarrow H^{p*}_{\mathit{eq}}(C^{\otimes p}), 
\end{equation}
which becomes additive after multiplying by $t$.
\end{lemma}

\begin{proof}
Since $x^{\otimes p}$ is a $(\bZ/p)$-invariant cocycle in $C^{\otimes p}$ (note that the Koszul signs here are always trivial), it is also a $d_{\mathit{eq}}$-cocycle. 

The next step is to show that if we have two cohomologous cocycles, $x_1-x_2 = dz$, then $x_1^{\otimes p}$ and $x_2^{\otimes p}$ are cohomologous in $(C^{\otimes p})_{\mathit{eq}}$. It is enough to consider the case where $C$ is three-dimensional, with basis $(x_1,x_2,z)$; the general case then follows by mapping this $C$ into any desired complex. Take a one-dimensional complex $D$ with a single generator $y$, and the map $C \rightarrow D$ which takes both $x_k$ to $y$ (and maps $z$ to zero). This is clearly a quasi-isomorphism, and therefore induces a quasi-isomorphism $(C^{\otimes p})_{\mathit{eq}} \rightarrow (D^{\otimes p})_{\mathit{eq}}$. Under that quasi-isomorphism, both $x_1^{\otimes p}$ and $x_2^{\otimes p}$ go to $y^{\otimes p}$. Therefore, they must be cohomologous in $(C^{\otimes p})_{\mathit{eq}}$.

The additivity statement can be proved by an explicit formula: if we take
\begin{equation} \label{eq:x1x2}
(x_1+x_2)^{\otimes p} - x_1^{\otimes p} - x_2^{\otimes p}
\end{equation}
and expand it out, we get $2^p - 2$ monomials, which occur in free $(\bZ/p)$-orbits. Take one representative for each orbit, add them up, and multiply the outcome by $\theta$. This yields a cochain in $(C^{\otimes p})_{\mathit{eq}}$ whose boundary is $t$ times \eqref{eq:x1x2}, up to sign.
\end{proof}

Finally, we return to the example of $S$. Take the cellular chain complex and reverse its grading, to make it cohomological. Then, on $C_{-*}(S)_{\mathit{eq}}$ we have
\begin{align}
&
d_{\mathit{eq}}(P_0\, t^k) = 0, \;\;
d_{\mathit{eq}}(P_0\, t^k\theta) = 0,
\\ &
d_{\mathit{eq}}(P_0\, t^k) = 0, \;\;
d_{\mathit{eq}}(P_0\, t^k\theta) = 0,
\\ &
d_{\mathit{eq}}(\sigma^j L_1\, t^k) = (Q_0 - P_0) t^k - (\sigma^{j+1} L_1 - \sigma^j L_1) t^k\theta,
\\ &
d_{\mathit{eq}}(\sigma^j L_1\, t^k\theta) = (Q_0 - P_0) t^k\theta - (L_1 + \cdots + \sigma^{p-1}L_1) t^{k+1},
\\ &
d_{\mathit{eq}}(\sigma^j B_2\, t^k) = -(\sigma^{j+1} L_1 - \sigma^j L_1) t^k + (\sigma^{j+1} B_2 - \sigma^j B_2) t^k\theta,
\\ &
d_{\mathit{eq}}(\sigma^j B_2\, t^k\theta) = -(\sigma^{j+1} L_1 -\sigma^j L_1) t^k\theta + (B_2 + \cdots + \sigma^{p-1}B_2) t^{k+1}.
\end{align}
With $\bF_p$-coefficients, we have the following analogue of Lemma \ref{lemma:homologyrelations}, proved in the same way:

\begin{lemma} \label{eq:cohomology-relations}
The following cohomology relations hold in $C_{-*}(S)_{\mathit{eq}}$:
\begin{align}
& \label{eq:coh-localize-1}
(P_0 - Q_0) t^k \sim (B_2 + \sigma B_2 + \cdots + \sigma^{p-1}B_2) t^{k+1},
\\ \label{eq:coh-localize-2}
&
(P_0 - Q_0) t^k\theta \sim (B_2 + \sigma B_2 + \cdots + \sigma^{p-1}B_2) t^{k+1}\theta.
\end{align}
\end{lemma}

\section{Basic moduli spaces}
This section introduces the relevant moduli spaces of pseudo-holomorphic curves, in their most basic form. This means that we look at a version of the small quantum product, and one of its properties, the divisor equation. Like the previous section, this should be considered as a toy model which introduces some ideas that will recur in more complicated form later on.

\subsection{\label{subsec:moduli-spaces-of-holomorphic-curves}}
Let $M^{2n}$ be a weakly monotone closed symplectic manifold. Choose a Morse function $f$ and metric $g$, so that the associated gradient flow is Morse-Smale. Our terminology for stable and unstable manifolds is that $\mathrm{dim}(W^s(x)) = |x|$ is the Morse index, whereas $\mathrm{dim}(W^u(x)) = 2n - |x|$. 

\begin{definition}
Fix some compatible almost complex structure $J$. A {\it $J$-holomorphic chain of length $l$} is a set of maps $$u_1, \dots, u_l: \bC P^1 \rightarrow M$$ such that $\overline{\partial}_{\mathit{J}} u = 0$, and such that 
\begin{equation}
u_k(\infty) = u_{k+1}(0) \text{ for } k=1,\dots,l-1.
\end{equation}
We call such a chain {\it simple} if each of the maps is simple (non-multiply-covered and non-constant) and no two of the maps are reparametrisations of each other. Two simple chains are called equivalent if they are related by reparametrisations $(\phi_1,\dots,\phi_l)$ of each component, such that $\phi_k(0) = 0$ and $\phi_k(\infty) = \infty$. The moduli space of simple chains representing some class $A \in H_2(M;\bZ)$ is denoted by $\scrM_A(\mathit{chain},l)$. It comes with evaluation maps at the ``endpoints of the chains'', which send $(u_1,\dots,u_l)$ to $u_1(0)$ and $u_l(\infty)$, respectively.
\end{definition}

\begin{assumption} \label{th:regular-1}
We fix some compatible almost complex structure $J$ with the following properties.
\begin{itemize} \itemsep.5em
\item[(i)]
All spaces $\scrM_A(\mathit{chain},l)$ are regular.
\item[(ii)]
On those spaces, the evaluation maps $(u_1,\dots,u_l) \mapsto u_1(0)$ are transverse to the stable and unstable manifolds of our Morse function. 
\end{itemize}
\end{assumption}

Assumption \ref{th:regular-1} is satisfied for generic choice of $J$. The simplest aspect is the $l = 1$ case of (i), which is just generic regularity of simple $J$-holomorphic spheres (because of the weak monotonicity condition, this also implies the absence of spheres with negative Chern number). The general form of (i) is a version of \cite[Definition 6.2.1]{mcduff-salamon} (using chains rather than general trees), and is generically satisfied by \cite[Theorem 6.2.6]{mcduff-salamon}. The transversality theory for evaluation maps developed there also yields the genericity of (ii).

Our main moduli space uses a specific $(p+2)$-marked sphere as the domain. We introduce specific notation for it: taking $\zeta^{1/2} = e^{\pi i/p}$, set
\begin{equation} \label{eq:c-curve}
\begin{aligned}
& C = \bC P^1, \\[-.5em] &
z_{C,0} = 0,\; z_{C,1} = \zeta^{1/2},\; z_{C,2} = \zeta^{3/2},\; \dots,\; z_{C,p} = \zeta^{(2p-1)/2} = \zeta^{-1/2}, \; z_{C,\infty} = \infty.
\end{aligned}
\end{equation}
An inhomogeneous term is a $J$-complex anti-linear vector bundle map $\nu_C: TC \rightarrow TM$, where both bundles involved have been pulled back to $C \times M$, such that $\nu_C$ is zero near the marked points \eqref{eq:c-curve}. The associated inhomogeneous Cauchy-Riemann equation is
\begin{equation} \label{eq:cauchy-riemann}
\begin{aligned} & u: C \longrightarrow M, \\
& (\bar\partial_J u)_z = \nu_{C,z,u(z))}.
\end{aligned}
\end{equation}
Given critical points $x_0,\dots,x_p,x_\infty$ of $f$, we consider solutions of \eqref{eq:cauchy-riemann} with incidence conditions at the (un)stable manifolds:
\begin{equation} \label{eq:incidence}
u(z_{C,0}) \in W^u(x_0),\; \dots,\; u(z_{C,p}) \in W^u(x_p),\;
u(z_{C,\infty}) \in W^s(x_{\infty}).
\end{equation}
It is maybe better to think of this as having gradient half-flowlines 
\begin{equation} \label{eq:half-flow-lines}
\begin{aligned}
& y_0,\dots,y_p: (-\infty,0] \longrightarrow M, \\
& y_k' =  \nabla f(y_k), \\
& y_k(0) = u(z_{C,k}), \\
& \textstyle\lim_{s \rightarrow -\infty} y_k(s) = x_k
\end{aligned}
\qquad \text{and} \qquad
\begin{aligned}
& y_\infty: [0,\infty) \longrightarrow M, \\
& y_\infty' = \nabla f(y_\infty), \\
& y_\infty(0) = u(z_{C,\infty}), \\
& \textstyle\lim_{s \rightarrow \infty} y_\infty(s) = x_\infty.
\end{aligned}
\end{equation}

\begin{assumption} \label{th:regular-2}
We impose the following requirements:
\begin{itemize} \itemsep.5em
\item[(i)] The moduli space of solutions of \eqref{eq:cauchy-riemann}, \eqref{eq:incidence} is regular.

\item[(ii)] Take an element in the same space, with a simple $J$-holomorphic bubble attached at an arbitrary point. This means that we have a pair $(u,u_0)$ with $u$ as in \eqref{eq:cauchy-riemann}, \eqref{eq:incidence}, a point $z \in C$, and a simple $J$-holomorphic $u_0: \bC P^1 \rightarrow M$ with $u(z) = u_0(0)$. We want this moduli space to be regular as well.

\item[(iii)] Consider solutions with a simple holomorphic chain attached at each of a subset of the $(p+2)$ marked points, and incidence constraints transferred accordingly. For simplicity, let's spell out what this means only in the case of a single chain, attached at $z_{C,\infty}$. In that case, we have a solution of \eqref{eq:cauchy-riemann}, and a simple holomorphic chain $(u_1,\dots,u_l)$, with the conditions
\begin{equation}
\begin{aligned}
& u(z_{C,0}) \in W^u(x_0),\;\dots,\; u(z_{C,p}) \in W^u(x_p), \\
& u(z_{C,\infty}) = u_1(0), \; u_l(\infty) \in W^s(x_\infty).
\end{aligned}
\end{equation}
We require that the resulting moduli space should be regular. In the general case where there are several marked points with a chain attached to each, we transfer the adjacency condition involving (un)stable manifolds to the end of the respective chain.
\end{itemize}
\end{assumption}

This assumption are satisfied for a generic choice of inhomogeneous term (where $J$ is assumed chosen as in Assumption \ref{th:regular-1}), following the argument from \cite[Chapter 8]{mcduff-salamon}. A few comments may be appropriate. In (ii), the bubble may be attached at one of the marked points. Let's say that this point is $z_{C,\infty}$, in which case we have
\begin{equation}
u(z_{C,\infty}) = u_0(0) \in W^s(x_\infty).
\end{equation}
Assumption \ref{th:regular-1}(ii), for $l = 1$, says that the subspace of maps $u_0$ satisfying $u_0(0) \in W^s(x_\infty)$ is regular. What we want to achieve is that the evaluation map on that subspace is transverse to $u \mapsto u(z_{C,\infty})$. This is clearly satisfied for generic $\nu_C$. In the same way, genericity of (iii) depends on Assumption \ref{th:regular-1}(ii), but this time for arbitrary $l$.

Given $A \in H_2(M;\bZ)$, let $\scrM_A(C,x_0,\dots,x_p,x_\infty)$ be the space of solutions of \eqref{eq:cauchy-riemann}, \eqref{eq:incidence} such that $u$ represents $A$. Given our regularity requirement, this is a manifold of dimension
\begin{equation} \label{eq:virtual-dimension}
\mathrm{dim}\, \scrM_A(C,x_0,\dots,x_p,x_\infty) = 2c_1(A)+|x_\infty|-|x_0|-\cdots-|x_p|.\end{equation}
We denote by $\bar\scrM_A(C,x_0,\dots,x_p,x_\infty)$ the standard compactification. On the pseudo-holomorphic map side, this involves the stable map compactification, and on the Morse-theoretic side one allows the flow lines to break. Details are in \cite[Section 5]{schwarz} (for illustration, see Figure $3$ there). To make the exposition more self-contained, we recall here that a point of the compactification consists of:

\begin{itemize} \itemsep.5em
\item A genus zero nodal Riemann surface $\hat{C}$ with $(p+2)$ smooth marked points $z_{\hat{C},0},\dots,z_{\hat{C},\infty}$. One of the irreducible components of that surface is distinguished, and identified with $C$ in a preferred way. Moreover, if one collapses all the other components (usually called bubble components), and transfers the marked points along with the collapse, then those marked points will end up in the same positions as in \eqref{eq:c-curve}. In other words, if $z_{\hat{C},k}$ does not lie on the distinguished component, then it must lie on a bubble tree attached to that component at $z_{C,k}$.
\item A map $\hat{u}: \hat{C} \rightarrow M$ which, on the distinguished component, is a solution of \eqref{eq:cauchy-riemann}, and on the other components, is a $J$-holomorphic map. Moreover, those $J$-holomorphic maps must be stable (if they are constant on some non-distinguished component, then that component must have at least three special points). Finally, the map $\hat{u}$ still represents the homology class $A$.
\item For each $k \in \{0,\dots,p\}$, a finite sequence of gradient flow lines $\hat{y}_{k,0}: \bR \rightarrow M, \dots, \hat{y}_{k,m_k-1}: \bR \rightarrow M$, $\hat{y}_{k,m_k}: (-\infty,0] \rightarrow M$ (all but the last should be non-constant). These should satisfy $$\begin{array}{l}
    \lim_{s \rightarrow -\infty} \hat{y}_{k,0}(s) = x_k, \\ 
    \lim_{s \rightarrow +\infty} \hat{y}_{k,j}(s) = \lim_{s \rightarrow -\infty} \hat{y}_{k,j+1}(s),  
    \\
    \hat{y}_{k,m_k}(0) = \hat{u}(z_{\hat{C},k}) 
\end{array}$$
\item Similarly, gradient flow lines $\hat{y}_{\infty,0}: [0,\infty) \rightarrow M$, $\hat{y}_{\infty,1}: \bR \rightarrow M$, \dots, $\hat{y}_{\infty,m_\infty}: \bR \rightarrow M$. Here, the conditions are that $$\begin{array}{l}
    \hat{y}_{\infty,0}(0) = \hat{u}(z_{\hat{C},\infty}), \\
    \lim_{s\rightarrow +\infty} \hat{y}_{\infty,j}(s) = \lim_{\rightarrow -\infty} \hat{y}_{\infty,j+1}(s), \\
    \lim_{s \rightarrow +\infty} \hat{y}_{\infty,m_\infty}(s) = x_\infty.
\end{array}$$
\end{itemize}  


\begin{lemma} \label{th:smooth-1}
(i) If the dimension \eqref{eq:virtual-dimension} is $0$, we have 
\begin{equation}
\scrM_A(C,x_0,\dots,x_p,x_\infty) = \bar\scrM_A(C,x_0,\dots,x_p,x_\infty), 
\end{equation}
which means that the moduli space is a finite set. 

(ii) If the dimension is $1$, the compactification is a manifold with boundary, with the interior being the space $\scrM_A(\cdots)$; the boundary points involve no bubbling, and only once-broken gradient flow lines.
\end{lemma}

\begin{proof}[Sketch of proof]
The proof is in \cite[Theorem 3.4]{schwarz} for the $0$-dimensional case, and \cite[Section 3.3]{schwarz} for the $1$-dimensional case. We will summarize it here. Recall that when compactifying the moduli space, what can occur is a mixture of Gromov compactification and breaking of Morse flowlines. Take a limit point in the form discussed above, assuming for simplicity that there is no breaking of Morse flow lines ($m_0 = \cdots = m_k = m_\infty = 0$). Collapse all the bubble components which carry constant $J$-holomorphic maps (called ghost components). Then carry out the following further simplifications:
\begin{itemize} \itemsep.5em
\item Suppose that after that initial collapse of constant components, all marked points come to lie on the distinguished component. In that case, we forget all bubbles except for one, which carries a nonconstant $J$-holomorphic map that intersects the image of the distinguished component at some point (these must be such a bubble). Finally, we also replace the map on that bubble component by its underlying simple map. That puts us in the situation of Assumption \ref{th:regular-2}(ii), where $(u,u_0)$ represents some class whose Chern number is less than equal that of $A$.

\item Take the other case (after the initial collapse, at least one marked point does not lie on the distinguished component). In that case, we forget any bubble tree that carries no marked points. This leaves only the distinguished component and at most one holomorphic chain attached at each of its $(p+2)$ marked points; the component of that chain which is furthest from the distinguished component will carry the marked point. As before, we replace all multiply covered bubbles with the underlying simple maps. Moreover, if two holomorphic maps $u_i,u_j$ with $i<j$ in one chain are reparametrisations of each other, we remove the bubbles carrying $u_i,\dots,u_{j-1}$. After this, all attached holomorphic chains are simple, and we are in the situation of Assumption \ref{th:regular-2}(iii), with at least one nontrivial bubble chain, and where again the Chern number has not increased from that of the original $A$.
\end{itemize}
All these simplified limits have codimension $\geq 2$, hence cannot occur in the moduli spaces under consideration. The case that includes Morse-theoretic breaking is similar, and we will not discuss it further.
\end{proof}
\begin{figure}
\begin{centering}
\includegraphics[scale = 0.8]{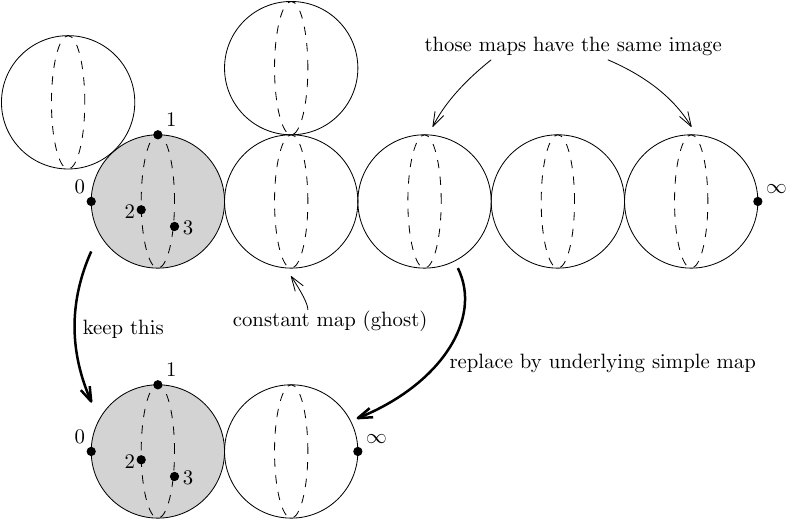}
\caption{Simplification process from the proof of Lemma \ref{th:smooth-1}, for $p = 3$. The stable map at the top (with $7$ components, and where the principal component is shaded) yields a solution of \eqref{eq:cauchy-riemann} with a length $1$ simple chain attached.}
\label{fig:cancel}
\end{centering}
\end{figure}

Given some coefficient field $\bF$, we denote by $\bF_x$ the one-dimensional vector space generated by orientations of $W^s(x)$, where the sum of the two orientations is zero. The Morse complex is
\begin{equation}
\mathit{CM}^k(f) = \bigoplus_{|x| = k} \bF_x.
\end{equation}
A choice of orientations of $W^s(x_0),\dots,W^s(x_p), W^s(x_\infty)$ determines an orientation of the moduli space $\scrM_A(C,x_0,\dots,x_p,x_\infty)$. In particular, every point in a zero-dimensional moduli space gives rise to a preferred isomorphism (an abstract version of a $\pm 1$ contribution)
$\bF_{x_0} \otimes \cdots \otimes \bF_{x_p} \iso \bF_{x_\infty}$. One adds up those contributions to get a map
\begin{equation}
m_A(C,x_0,\dots,x_p,x_\infty): \bF_{x_0} \otimes \cdots \otimes \bF_{x_p} \longrightarrow \bF_{x_\infty},
\end{equation}
and those maps are the coefficients of a chain map 
\begin{equation} \label{eq:qs-map}
S_A: \mathit{CM}^*(f)^{\otimes p+1} \longrightarrow \mathit{CM}^{*-2c_1(A)}(f). 
\end{equation}
Up to chain homotopy, this map is independent of the choice of almost complex structure and inhomogeneous term, by a parametrized version of our previous argument. Of course, the outcome is not in any sense surprising:

\begin{lemma} \label{th:p-plus-one-fold}
Up to chain homotopy, $S_A(x_0,x_1,\dots,x_p)$ is the $A$-contribution to the $(p+1)$-fold quantum product $x_0 \ast x_1 \ast \cdots \ast x_p$.
\end{lemma}
\begin{proof}
This is a familiar argument, which involves degenerating $C$ to a nodal curve each of whose components has three marked points, one option being that drawn in Figure \ref{fig:multideg}(i); each component will again carry a Cauchy-Riemann equation with an inhomogeneous term. In our Morse-theoretic context, there is an additional step, familiar from the proof that the PSS map is an isomorphism, such as in \cite[Theorem 6]{katic-milinkovic} (see Fig 2), \cite[Section 4]{albers} (see Fig 6), \cite[Section 4]{lu}, or for more details \cite{katic}. Namely, one adds a length parameter, and inserts a finite length flow line of our Morse function at each node. As the length goes to infinity, each of the flow lines we have inserted breaks, see Figure \ref{fig:multideg}(ii); and that limit gives rise to the Morse homology version of the iterated quantum product. The parametrized moduli space (consisting of, first, the parameter used to degenerate $C$; and then in the second step, using the finite edge-length as a parameter) then yields a chain homotopy between those two operations.
\end{proof}
\begin{figure}
\begin{centering}
\includegraphics{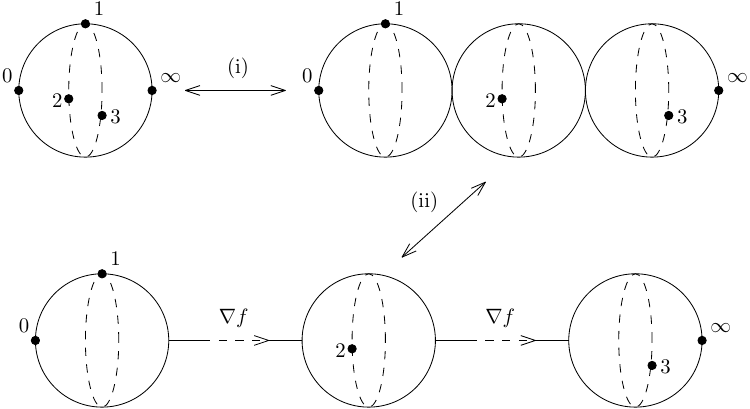}
\caption{\label{fig:multideg}A schematic picture of the proof of Lemma \ref{th:p-plus-one-fold}, with $p = 3$.}
\end{centering}
\end{figure}

\begin{remark} \label{th:negative-energy}
Our use of inhomogeneous terms means that the moduli space could be nonzero for classes $A \in H_2(M;\bZ)$ which do not give rise to monomials in $\Lambda$ (because $\int_A \omega_M$ is either negative, or it's zero but $A \neq 0$). However, by choosing the inhomogeneous term small and using a compactness argument, one can rule out that undesired behaviour for any specific $A$. Since the outcome is independent of the choice up to chain homotopy, the resulting cohomology level structure is indeed defined over $\Lambda$.
\end{remark}

\subsection{\label{subsec:augmented-moduli-spaces-of-holomorphic-curves}}
Fix an oriented codimension $2$ submanifold $\Omega \subset M$. When choosing an almost complex structure, there are additional restrictions:

\begin{assumption} \label{th:regular-1b}
In the situation of Assumption \ref{th:regular-1}, we additionally require that the evaluation map on the space of simple $J$-holomorphic chains should be transverse to $\Omega$.
\end{assumption}

We equip the Riemann surface \eqref{eq:c-curve} with ``an additional marked point which can move freely'' (and which will carry an $\Omega$-incidence constraint). Formally, this means that we consider a family of genus zero nodal curve with sections 
\begin{equation} \label{eq:s2-family}
\begin{aligned}
& \scrC \longrightarrow S, \\
& z_{\scrC,0},\dots,z_{\scrC,p},z_{\scrC,\infty},z_{\scrC,*}: S \longrightarrow \scrC
\end{aligned}
\end{equation}
where the parameter space $S$ is again a copy of $\bC P^1$, and such that the following holds:
\begin{itemize} \itemsep.5em
\item
The critical values of \eqref{eq:s2-family} are precisely the marked points from \eqref{eq:c-curve}. If $v$ is a regular value, the fibre $\scrC_v$ is canonically identified with $C$; that identification takes the points $z_{\scrC_r,0},\dots,z_{\scrC_r,p},z_{\scrC_r,\infty}$ arising from \eqref{eq:s2-family} to their counterparts in \eqref{eq:c-curve}, and the remaining point $z_{\scrC_v,*}$ to $v$.

\item
If $v$ is a singular value, $\scrC_v = \scrC_{v,+}\, \cup\, \scrC_{v,-}$ is a nodal surface with two components. The first component $\scrC_{v,+}$ is again identified with $C$, and the second component $\scrC_{v,-}$ is a rational curve attached to the first one at $v$. The first component carries all the marked points that $C$ does, with the exception of the one which is equal to $v$; and the second component carries the two remaining marked points, considered to be distinct and also different from the node (so, the second component has three special points, which identifies it up to unique isomorphism).
\end{itemize}
Explicitly, \eqref{eq:s2-family} is constructed by starting with the trivial family $C \times S \rightarrow S$, and then blowing up the points $(v,v)$, where $v$ is one of the marked points in \eqref{eq:c-curve}. One takes the proper transforms of the constant sections and of the diagonal section, which yield the $z_{\scrC}$'s from \eqref{eq:s2-family}. 

Denote by $\scrC^{\mathit{sing}} \subset \scrC$ the set of $(p+2)$ nodes, and by $\scrC^{\mathit{reg}}$ its complement. We write $T(\scrC^{\mathit{reg}}/S)$ for the fibrewise tangent bundle, which is a complex line bundle on $\scrC^{\mathit{reg}}$.
A fibrewise inhomogeneous term on $\scrC$ is a complex anti-linear map $\nu_{\scrC}: T(\scrC^{\mathit{reg}}/S) \rightarrow TM$, where both bundles involved have been pulled back to $\scrC^{\mathit{reg}} \times M$, and with the property that $\nu_{\scrC}$ is zero outside a compact subset (meaning, in a neighbourhood of $\scrC^{\mathit{sing}} \times M \subset \scrC \times M$). Suppose that we have chosen such a term. One can then consider the moduli space of pairs $(v,u)$, where
\begin{equation} \label{eq:parametrized-cauchy-riemann}
\begin{aligned}
& v \in S, \; u: \scrC_v \longrightarrow M, \\
& (\bar\partial_J u)_z = \nu_{\scrC_v,z,u(z)}.
\end{aligned}
\end{equation}
In the case where $\scrC_v$ has a node, the second equation is imposed separately on each of its components (with the assumption that both preimages of the node must be mapped to the same point, so as to constitute an actual map on $\scrC_v$). This makes sense since, near each of the preimages of the node, the equation reduces to the ordinary $J$-holomorphic curve equation. The incidence conditions are
\begin{equation} \label{eq:added-incidence}
u(z_{\scrC_v,0}) \in W^u(x_0),\; \dots,\; u(z_{\scrC_v,p}) \in W^u(x_p),\;
u(z_{\scrC_v,\infty}) \in W^s(x_{\infty}), \; u(z_{\scrC_v,*}) \in \Omega.
\end{equation}

\begin{assumption} \label{th:regular-2b}
We impose the following requirements:
\begin{itemize} \itemsep.5em
\item[(i)] 
The space of all solutions of \eqref{eq:parametrized-cauchy-riemann}, \eqref{eq:added-incidence} should be regular. This should be understood as two distinct conditions: on the open set of regular values $v$, regularity holds in the parametrized sense; and for each singular value $v$, it holds in the ordinary unparametrized sense. 

\item[(ii)] Take an element $(v,u)$ in the same space, with a simple $J$-holomorphic bubble attached at an arbitrary point, in the same sense as in Assumption \ref{th:regular-2}(ii) (the attaching point can be a marked point, or even the node if $v$ is singular). Then, that moduli space should be regular as well. As in (i), this should be interpreted as two different conditions, depending whether $v$ is regular or not.

\item[(iii)] Consider solutions for regular $v$, which have a simple holomorphic chain attached at a subset of the $(p+3)$ marked points, and where the incidence constraint has been transferred to the end of that chain, as in Assumption \ref{th:regular-2}(iii). Then, the resulting moduli space should again be regular.

\item[(iv)] Take a singular $v$, We look at a situation similar to (iii), but where additionally, there may be a simple holomorphic chain separating the two components of $\scrC_v$. Let's spell out what that means (ignoring the possible existence of chains at the marked points). Write $z_{\pm} \in \scrC_{v,\pm}$ for the preimages of the node. In the definition of the moduli space \eqref{eq:parametrized-cauchy-riemann}, the $\scrC_{v,\pm}$ carry maps $u_{\pm}$ which necessarily satisfy $u_-(z_-) = u_+(z_+)$. However, in our limiting situation, we instead have a simple chain $(u_1,\dots,u_l)$ such that
\begin{equation}
u_-(z_-) = u_1(0), \; u_+(z_+) = u_l(\infty).
\end{equation}
Again, we require that the resulting space should be regular.
\end{itemize}
\end{assumption}

As before, given $A \in H_2(M;\bZ)$, we write $\scrM_A(\scrC,x_0,\dots,x_\infty,\Omega)$ for the space of solutions of \eqref{eq:parametrized-cauchy-riemann}, \eqref{eq:added-incidence} representing $A$. The added parameter $v \in S$ compensates exactly for the evaluation constraint at $z_{\scrC,*}$, so that we get the same expected dimension as before,
\begin{equation} \label{eq:virtual-dimension-2}
\mathrm{dim}\, \scrM_A(\scrC,x_0,\dots,x_p,x_\infty,\Omega) = 2c_1(A)+|x_\infty|-|x_0|-\cdots-|x_p|.
\end{equation}
Concerning the analogue of the stable map compactification, we have a version of Lemma \ref{th:smooth-1} (with essentially the same proof):

\begin{lemma} \label{th:smooth-2}
(i) If the dimension \eqref{eq:virtual-dimension-2} is $0$, we have a finite set
\begin{equation}
\scrM_A(\scrC,x_0,\dots,x_p,x_\infty,\Omega) = \bar\scrM_A(\scrC,x_0,\dots,x_p,x_\infty,\Omega).
\end{equation}

(ii) If the dimension is $1$, the compactification is a manifold with boundary, with the boundary points only involving once-broken gradient flow lines. 

In both cases (i) and (ii), the moduli space and its compactification contain only points where $v$ is a regular value.
\end{lemma}

We define $m_A(\scrC,x_0,\dots,x_p,x_\infty,\Omega)$ to be the signed count of points in the zero-dimensional moduli spaces. As before, one can assemble these into a chain map
\begin{equation} \label{eq:qp-map}
P_{A,\Omega}: \mathit{CM}^*(f)^{\otimes p+1} \longrightarrow \mathit{CM}^{*-2c_1(A)}(f).
\end{equation}
Up to chain homotopy, this is independent of the choices of $J$ and $\nu_{\scrC}$, and also depends only on $[\Omega] \in H^2(M;\bZ)$.

\subsection{}
The remaining topic in this section is the analogue of the divisor axiom. As one would expect, this is not particularly difficult, but requires a bit of technical discussion around forgetting a marked point. For the submanifold $\Omega$, we want to assume that it is transverse to the stable and unstable manifolds of the Morse function.

\begin{lemma} \label{th:smooth-1b}
In the situation of Lemma \ref{th:smooth-1}, the following holds generically: any map $u$ in a zero-dimensional space $\scrM_A(C,x_0,\dots,x_p,x_\infty)$ intersects $\Omega$ transversally, and moreover, all those intersections happen away from the marked points. The same is true within the smaller space of those $\nu_C$ which vanish close to the marked points.
\end{lemma}

This is standard (transversality of evaluation maps). The only wrinkle specific to our case is that the intersections avoid the marked points: but if they didn't, we would have an incidence constraint with $\Omega \cap W^s(x)$ or $\Omega \cap W^u(x)$, and those can be ruled out for dimension reasons.

\begin{proposition} \label{th:divisor-equations-proposition}
Fix some $A$. For suitable choices made in the definitions, the maps \eqref{eq:qs-map} and \eqref{eq:qp-map} are related by $P_{A,\Omega} = (A \cdot \Omega)\,S_A$. (For arbitrary choices, the same relation will therefore hold up to chain homotopy.)
\end{proposition}

\begin{proof}
Even more explicitly, our statement says that one can arrange that
\begin{equation} \label{eq:strict-divisor}
m_A(\scrC,x_0,\dots,x_p,x_\infty,\Omega) = (A \cdot \Omega)\, m_A(C,x_0,\dots,x_p,x_\infty).
\end{equation}
We start with $J$ as in Assumption \ref{th:regular-1b}, and a $\nu_C$ as in Lemma \ref{th:smooth-1b}. Because the inhomogeneous term is zero near the marked points, it can be pulled back to give a fibrewise inhomogeneous term $\nu_{\scrC}$. To clarify, if $\scrC_v$ is a singular fibre, then $\nu_{\scrC_v}$ is supported on $\scrC_{v,+} \iso C$, and zero on the other component $\scrC_{v,-}$. Let's consider the structure of the resulting moduli spaces. Given a point in the compactification $\bar\scrM_A(\scrC,x_0,\dots,x_p,x_\infty,\Omega)$, one can forget the position of the $\ast$ marked point, and then collapse unstable components (components which are not $C$, and which carry a constant $J$-holomorphic map and less than three special points). The outcome is a (continuous) map
\begin{equation} \label{eq:forget}
\bar\scrM_A(\scrC,x_0,\dots,x_p,x_\infty,\Omega) \longrightarrow \bar\scrM_A(C,x_0,\dots,x_p,x_\infty).
\end{equation}
Now suppose that the dimension is zero. Then, the target in \eqref{eq:forget} is $\scrM_A(C,x_0,\dots,x_p,x_\infty)$, and consists only of maps $u: C \rightarrow M$ whose intersection points with $\Omega$ are not marked points. The preimage of $u$ under \eqref{eq:forget} is necessarily an element of $\scrM_A(\scrC,x_0,\dots,x_p,x_\infty,\Omega)$, with $v$ a regular value; 
such preimages correspond bijectively to points in $u^{-1}(\Omega)$, hence form a finite set, and (because of the transversality condition in Lemma \ref{th:smooth-1b}) are regular points in the parametrized moduli space. Finally, the sign of their contribution to $m_A(\scrC,x_0,\dots,x_p,x_\infty,\Omega)$ is given by multiplying the contribution of $u$ to $m_A(C,x_0,\dots,x_p,x_\infty)$ with the local intersection number (sign) of $u$ and $\Omega$ at the relevant point.

We have now shown that $\scrM_A(\scrC,x_0,\dots,x_p,x_\infty,\Omega) = \bar\scrM_A(\scrC,x_0,\dots,x_p,x_\infty,\Omega)$ is regular, and that counting points in it exactly yields the right hand side of \eqref{eq:strict-divisor}. The $\nu_\scrC$ used for this purpose may not satisfy Assumption \ref{th:regular-2b}, so this setting is not strictly speaking part of our general definition of $m_A(\scrC,x_0,\dots,x_p,x_\infty,\Omega)$. However, we can find a small perturbation of $\nu_\scrC$ which does satisfy Assumption \ref{th:regular-2b}, and points in the associated zero-dimensional moduli spaces will correspond bijectively to those for the original $\nu_\scrC$, because of the compactness and regularity of the original space.
\end{proof}

\section{Quantum Steenrod operations}
This section concerns the operations \eqref{eq:sigma} and \eqref{eq:modified-sigma}. We first set up the various equivariant moduli spaces, then define $Q\Sigma_b$, and discuss its properties. Then we proceed to do the same for $Q\Pi_{a,b}$, and go as far as establishing \eqref{eq:pi-divisor}.

\subsection{\label{subsec:equivariant-moduli-space}}
We equip $C = \bC P^1$ with the $(\bZ/p)$-action generated by the same rotation as in Section \ref{sec:subsec-basic-top-b}, but here denoted by $\sigma_C$. Fix a compatible almost complex structure $J$. An equivariant inhomogeneous term $\nu_C^{\mathit{eq}}$ is a smooth complex-antilinear map $TC \rightarrow TM$, where both bundles have been pulled back to $S^\infty \times_{\bZ/p} C \times M$, and with the same condition of vanishing near the marked points as before. More concretely, one can think of it as a family $\nu^{\mathit{eq}}_{C,w}$ of inhomogeneous terms (in the standard sense) parametrized by $w \in S^\infty$, with the property that
\begin{equation} \label{eq:nu-equivariance}
\nu^{\mathit{eq}}_{C,\tau(w),z,x} = \nu^{\mathit{eq}}_{C,w,\sigma_C(z),x} \circ D\sigma_z: \mathit{TC}_z \rightarrow \mathit{TM}_x
\;\; \text{ for $(w,z,x) \in S^\infty \times C \times M$.}
\end{equation}
Such equivariant data always exist, because the $\bZ/p$-action on the space $S^{\infty} \times C \times M$ is free. 
Consider the following parametrized moduli problem:
\begin{equation} \label{eq:conditionsformodulispace}    
\begin{aligned}
& w \in S^\infty, \;\; u: C \longrightarrow M, \\ 
& (\bar\partial_J u)_z = \nu^{eq}_{C,w,z,u(z))}. \\
\end{aligned}
\end{equation}
Note that this inherits a $(\bZ/p)$-action, generated by
\begin{equation} \label{eq:flip-u}
(w,u) \longmapsto (\tau(w), u \circ \sigma_C).
\end{equation}
Fix critical points $x_0,\dots,x_p,x_\infty$, and impose the same incidence constraints as in \eqref{eq:incidence} or equivalently \eqref{eq:half-flow-lines}. Moreover, we fix an integer $i \geq 0$ and use that to restrict the parameter $w$ to one of the cells from \eqref{eq:d-delta-1}, \eqref{eq:d-delta-2}. More precisely, the condition is that
\begin{equation} \label{eq:cellrestriction}
w \in \Delta_i \setminus \partial \Delta_i \subset S^\infty.
\end{equation}
Take solutions of \eqref{eq:conditionsformodulispace}, \eqref{eq:incidence}, \eqref{eq:cellrestriction} that represent some class $A \in H_2(M;\bZ)$, and denote the resulting moduli space by $\scrM_A(\Delta_i \times C,x_0,\dots,x_p,x_\infty)$. The expected dimension increases by the number of parameters,
\begin{equation} \label{eq:equivariant-moduli-space-dimension}
\mathrm{dim}\, \scrM_A(\Delta_i \times C, x_0,\dots,x_p,x_\infty) = 
i + 2c_1(A) + |x_\infty| - |x_0| - \cdots - |x_p|.
\end{equation}
Note that while one could define such moduli spaces for more general cells $\tau^j(\Delta_i)$, that is redundant because of \eqref{eq:flip-u}. To express that more precisely, write $(x_1^{(j)},\dots,x_p^{(j)})$ for the $p$-tuple obtained by cyclically permuting $(x_1,\dots,x_p)$ $j$ times (to the right, so $x_1^{(1)} = x_p$). Then,
\begin{equation} \label{eq:shift-spaces}
\begin{aligned}
& \scrM_A(\tau^j(\Delta_i) \times C, x_0,\dots,x_p,x_\infty) \stackrel{\iso}{\longrightarrow}
\scrM_A(\Delta_i \times C, x_0, x_1^{(j)},\dots,x_p^{(j)},x_\infty), \\
& (w,u) \longmapsto (\tau^{-j}(w), u \circ \sigma_C^{-j}).
\end{aligned}
\end{equation}
There is also a natural compactification, denoted by $\bar\scrM_A(\cdots)$ as usual. This combines the (parametrized) stable map compactification, breaking of Morse flow lines, and instances where the parameter $w$ reaches the boundary of $\Delta_i$. 

\begin{lemma} \label{th:parametrized-space}
For generic $J$ and $\nu_C^{\mathit{eq}}$, the following properties are satisfied. 

(i) If the dimension \eqref{eq:equivariant-moduli-space-dimension} is zero, we get a finite set
\begin{equation}
\scrM_A(\Delta_i \times C, x_0,\dots,x_p,x_\infty) = \bar\scrM_A(\Delta_i \times C, x_0,\dots,x_p,x_\infty).
\end{equation}

(ii) If the dimension is $1$, the moduli space is regular, and its compactification is a manifold with boundary. Besides the usual boundary points arising from broken Morse flow lines, one has solutions $(w,u)$ where $w \in \partial \Delta_i$. Using \eqref{eq:shift-spaces}, the set of such boundary points can be identified with a disjoint union
\begin{equation}
\bigcup_j\,
\scrM_A(\Delta_{i-1} \times C, x_0,x_1^{(j)},\dots,x_p^{(j)}, x_\infty)
\;\; \text{over }
\begin{cases} 
j = 0,\dots,p-1 & \text{$i$ even}, \\
j = 0,1 & \text{$i$ odd.}
\end{cases}
\end{equation}
\end{lemma}

In (ii), note that the only points $w \in \partial \Delta_i$ that occur lie in the interior of the cells of dimension $(i-1)$. In particular, the fact that the even-dimensional $\Delta_i$ have corners can be disregarded. The proof of Lemma \ref{th:parametrized-space} is simply a parametrized version of that of Lemma \ref{th:smooth-1}: one imposes Assumption \ref{th:regular-1} on $J$, and the parametrized analogue of Assumption \ref{th:regular-2} on $\nu_C^{\mathit{eq}}$, where the parameter space is taken to be each $\Delta_i \setminus \partial \Delta_i$. We will not discuss the argument further, and move ahead to its implications.

As usual, we count points in zero-dimensional moduli spaces, and collect those coefficients into
\begin{equation} \label{eq:sigma-i-a}
\Sigma_A(\Delta_i,\dots): \mathit{CM}^*(f) \otimes \mathit{CM}^*(f)^{\otimes p} \longrightarrow \mathit{CM}^{*-i-2c_1(A)}(f).
\end{equation}
Lemma \ref{th:parametrized-space}(ii), with the orientations of the $\Delta_i$ taken into account as in \eqref{eq:partial-delta-1}, \eqref{eq:partial-delta-2}, shows that, $d$ being the Morse differential,
\begin{equation} \label{eq:parametrized-boundary}
\begin{aligned}
& 
d \Sigma_A(\Delta_i, x_0,\dots,x_p) - (-1)^i \sum_{j=0}^p (-1)^{|x_0|+\cdots+|x_{j-1}|} \Sigma_A(\Delta_i,x_0,\dots,dx_j,\dots,x_p) 
\\ & 
= \begin{cases} \displaystyle
\sum_j (-1)^* \Sigma_A(\Delta_{i-1},x_0,x_1^{(j)},\dots,x_p^{(j)}) & \text{$i$ even,} \\
(-1)^* \Sigma_A(\Delta_{i-1},x_0,x_1^{(1)},\dots,x_p^{(1)}) - \Sigma_A(\Delta_{i-1},x_0,x_1,\dots,x_p) & \text{$i$ odd.}
\end{cases}
\end{aligned}
\end{equation}
Here, $(-1)^*$ is the Koszul sign associated with permuting $(x_1,\dots,x_p)$. 

\begin{remark}
Our sign conventions for parametrized pseudo-holomorphic map equations are as follows. Consider, just for the simplicity of notation, operations induced by a Cauchy-Riemann equation on the sphere, with one input and one output. If we have a family of such equations depending on a parameter space $\Delta$ which is a manifold with boundary, then the resulting endomorphism of $\mathit{CM}^*(f)$ satisfies
\begin{equation}
d\phi_\Delta - (-1)^{|\Delta|} \phi_\Delta d = \phi_{\partial \Delta}.    
\end{equation}
Note that this differs from the convention in \cite[Section 4c]{seidel-gaussmanin}; one can translate betwen the two by multiplying $\phi_\Delta$ with $(-1)^{|\Delta|(|\Delta|-1)}$. 
\end{remark}

From now on, we will exclusively work with coefficients in $\bF = \bF_p$.

\begin{lemma} \label{th:sigma-chain-map}
Suppose that $b$ is a Morse cocycle. Then, for each $i$ and $A$, 
\begin{equation} \label{eq:s-b}
x \longmapsto (-1)^{|b|\,|x|} \Sigma_A(\Delta_i,x,b,\dots,b)
\end{equation}
is a chain map (an endomorphism of the Morse complex) of degree $p|b|-i-2c_1(A)$.
\end{lemma}

This is immediate, by specializing \eqref{eq:parametrized-boundary} to $x_1 = \cdots = x_p = b$. In particular, in this case the Koszul signs in \eqref{eq:parametrized-boundary} are $1$: so for odd $i$, the expression on the right hand side vanishes, whereas for even $i$ that expression is $p \Sigma_A(\Delta_{i-1}, x_0,b,\dots,b)$, which vanishes modulo $p$.

We combine these operations into a series, which is a chain map
\begin{equation} \label{eq:multi-b}
\begin{aligned}
& \Sigma_{A,b}: \mathit{CM}^*(f) \longrightarrow (\mathit{CM}(f)[[t,\theta]])^{*+p|b|-2c_1(A)}, \\
& x \longmapsto (-1)^{|b|\,|x|} \sum_k \Big(\Sigma_A(\Delta_{2k},x,b,\dots,b) + (-1)^{|b|+|x|} \Sigma_A(\Delta_{2k+1},x,b,\dots,b) \theta\Big) t^k,
\end{aligned}
\end{equation}
One can also sum formally over all $A$ and extend the outcome $\Lambda$-linearly,
\begin{equation} \label{eq:lambda-linear}
\Sigma_b = \sum_A q^A \Sigma_{A,b}: \mathit{CM}^*(f;\Lambda) \longrightarrow \mathit{CM}^{*+p|b|}(f;\Lambda).
\end{equation}


\begin{lemma} \label{th:linearity}
Up to homotopy, \eqref{eq:multi-b} depends only on cohomology class $[b]$, and moreover, that dependence is linear.
\end{lemma}

\begin{proof}
Take $\mathit{CM}^*(f)^{\otimes p}$, with the $\bZ/p$-action given by cyclic permutation, and form the associated equivariant complex as in \eqref{eq:general-equivariant}. Consider the $t$-linear map
\begin{equation} \label{eq:equivariant-complex-map}
\begin{aligned}
& \Sigma_A^{\mathit{eq}}: \mathit{CM}^*(f) \otimes (\mathit{CM}^*(f)^{\otimes p})_{\mathit{eq}} \longrightarrow (\mathit{CM}(f)[[t,\theta]])^{*-2c_1(A)}, 
\\
& x_0 \otimes (x_1 \otimes \cdots \otimes x_p) \longmapsto
\\
& \qquad \sum_k \Big( \Sigma_A(\Delta_{2k},x_0,\dots,x_p) + (-1)^{|x_0|+\cdots+|x_p|} \Sigma_A(\Delta_{2k+1},x_0,\dots,x_p) \theta \Big) t^k,
\\
& x_0 \otimes (x_1 \otimes \cdots \otimes x_p)\, \theta \longmapsto 
 \sum_k \Big( \Sigma_A(\Delta_{2k},x_0,\dots,x_p) \theta \\
& \qquad \qquad \qquad \qquad -
(-1)^{|x_0|+\cdots+|x_p|} \sum_j 
j (-1)^* \Sigma_A(\Delta_{2k+1},x_0,x_1^{(j)},\dots,x_p^{(j)}) t \Big)t^k,
\end{aligned}
\end{equation}
where $(-1)^*$ is again the Koszul sign. The equation \eqref{eq:parametrized-boundary}, along with \eqref{eq:partial-delta-1} and \eqref{eq:partial-delta-2}, amounts to saying that \eqref{eq:equivariant-complex-map} is a chain map with respect to $d_{eq}$. As an elementary algebraic consequence, one has the following: if $c$ is any cocycle in $(\mathit{CM}^*(f)^{\otimes p})_{\mathit{eq}}$, then 
\begin{equation}
x \longmapsto (-1)^{|c|\,|x|} \Sigma_A^{\mathit{eq}}(x \otimes c)
\end{equation}
is an endomorphism of the chain complex $\mathit{CF}^*(f)$ of degree $|c|-2c_1(A)$. The homotopy class of that endomorphism depends only on the cohomology class of $c$. Moreover, they are additive in $c$. Applying that construction to $c = b \otimes \cdots \otimes b$ yields precisely \eqref{eq:multi-b}.

From Lemma \ref{lemma:power-map-additive-t}, we know that the cohomology class $[b \otimes \cdots \otimes b] \in H^*_{\mathit{eq}}(\mathit{CM}^*(f)^{\otimes p})$ only depends on that of $[b]$, which proves our first claim. By the same Lemma, if we use $t(b \otimes \cdots \otimes b)$ instead, the associated operation becomes linear in $[b]$. But that operation is just $t$ times \eqref{eq:multi-b}, so it follows that \eqref{eq:multi-b} itself must be linear in $[b]$.
\end{proof}

\begin{definition} \label{th:q-sigma}
For $b \in H^*(M;\bF_p)$ and $A \in H_2(M;\bZ)$, we define the operation $Q\Sigma_{A,b}$ from \eqref{eq:q-sigma-a} to be the cohomology level map induced by \eqref{eq:multi-b}. Correspondingly, \eqref{eq:lambda-linear} is the chain map underlying $Q\Sigma_b$.
\end{definition}

Here, we are implicitly using the fact that the chain level operations are independent of all choices up to chain homotopy. The proof is standard, using moduli spaces with one extra parameter, and will be omitted. Among the previously stated properties of $Q\Sigma$, \eqref{eq:non-equivariant} concerns the contribution of the cell $\Delta_0$, which is the operation from Section \ref{subsec:moduli-spaces-of-holomorphic-curves}, hence is exactly Lemma \ref{th:p-plus-one-fold}. The next two Lemmas correspond to \eqref{eq:classical-xi} and \eqref{eq:steenrod-sigma}.

\begin{lemma} \label{th:classical-sq}
For $A = 0$, $Q\Sigma_{A,b}$ is the cup product with $\mathit{St}(b)$.
\end{lemma}

\begin{proof}[Sketch of proof]
It will be convenient for this purpose to allow a slightly larger set of choices in the construction. Namely, we choose $s$-dependent vector fields for $s$ in the relevant half-line given below, parametrising either an ``incoming" or ``outgoing" flowline respectively 
\begin{equation} \label{eq:z-fields}
\begin{aligned}
& Z_{0,w,s},\dots,Z_{p,w,s} \in \smooth(TM) \text{ for $w \in S^\infty$, $s \leq 0$, with } \;\; Z_{k,w,s} = \nabla f \text{ if $s \ll 0$,}
\\
& Z_{\infty,w,s} \in \smooth(TM) \text{ for $w \in S^\infty$, $s \geq 0$, with } \;\; Z_{\infty,w,s} = \nabla f \text{ if $s \gg 0$.}
\end{aligned}
\end{equation} 
These are used to replace the gradient flow equations in \eqref{eq:half-flow-lines} by $dy_k/ds = Z_{w,k,s}$. The effect is that in the incidence conditions \eqref{eq:incidence}, the (un)stable manifolds are replaced by perturbed versions. In particular, the transversality of those incidence conditions imposed on pseudo-holomorphic curves can then be achieved by choosing \eqref{eq:z-fields} generically. This strategy (with minor technical differences) goes back to the Morse-theoretic definition of ordinary Steenrod operations in \cite[Section 2]{betz-cohen95}. In \cite[Appendix B.1]{wilkins18}, the iterative procedure to choose such $Z_{k,w,s}$ in a way that one obtains a moduli space cut-out transversely is given in detail, in addition to the fact that such a choice is generic. 
%

We impose an additional symmetry condition, which ensures that \eqref{eq:shift-spaces} still holds:
\begin{equation} \label{eq:z-symmetry}
Z_{k+1,w,s} = Z_{k,\tau(w),s} \;\; \text{ for $k = 1,\dots,p-1$}
\end{equation}
Considering just $A = 0$, this means that we can take the inhomogeneous term to be zero throughout, so that all maps $u$ are constant (and of course regular). The resulting moduli spaces are purely Morse-theoretical, see Figure \ref{fig:morse}(i) for a schematic representation. Without violating the symmetry property \eqref{eq:shift-spaces}, we can deform our moduli spaces as indicated in Figure \ref{fig:morse}(ii). This separates the coincidence condition at the endpoints of the half-flow lines into two parts, joined by a finite length flow line of some other auxiliary $s$-dependent vector field.  More precisely, we use the length as an additional parameter, and all vector fields involved may depend on that. One can arrange that as the length goes to $\infty$, the limit consists of split solutions as in Figure \ref{fig:morse}(iii), where the vector fields on the bottom part are independent on $w \in S^\infty$. It is now straightforward to see that this limit is the combination of the Morse-theoretic cup product and the Morse-theoretic version of the Steenrod operation \cite{betz-cohen95, cohen-norbury12}.
\end{proof}
\begin{figure}
\includegraphics{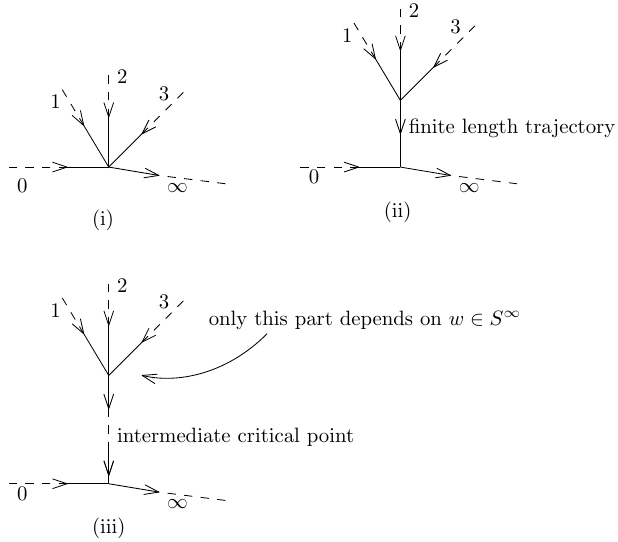}
\begin{centering}
\caption{\label{fig:morse}A schematic picture of the proof of Lemma \ref{th:classical-sq}.}
\end{centering}
\end{figure}

\begin{lemma}
$Q\Sigma_{A,b}(1)$ agrees with the $A$-contribution to the quantum Steenrod operation $Q\mathit{St}(b)$, as defined in \cite{fukaya93b, wilkins18} (for $p = 2$) or \cite{seidel19} (all $p$).
\end{lemma}

\begin{proof}[Sketch of proof]
Morse-theoretically, $1$ is represented by the sum of local minima of the Morse function. Hence, the associated incidence condition \eqref{eq:incidence} requires $u(0)$ to lie in an open dense set, and is generically satisfied on every zero-dimensional moduli space. In other words, $\Sigma_{A,b}(1)$ can be computed by forgetting the zero-th marked point and its incidence condition. The outcome is exactly the definition of the quantum Steenrod operation, generalizing the $p = 2$ case from \cite{wilkins18} in a straightforward way; compared to the slightly more abstract formulation in \cite[Section 9]{seidel19}, the only difference is that we stick to a specific cell decomposition of $B\bZ/p = S^\infty/(\bZ/p)$.
\end{proof}

\subsection{}
The final piece of our discussion of $Q\Sigma$ operations concerns \eqref{eq:compose-sigma}. We assume that the underlying cochain level map $\Sigma_b$ has been extended to $b \in \mathit{CM}^*(f) \otimes \Lambda$, as in \eqref{eq:b-variable-extension}.

\begin{proposition} \label{th:sigma-composition}
Fix Morse cocycles $b$ and $\tilde{b}$, and write $\tilde{b} \ast b \in \mathit{CM}^*(f) \otimes \Lambda$ for a cochain representative of their quantum product. Then, there is a chain homotopy
\begin{equation}
\Sigma_{\tilde{b}} \circ \Sigma_b \htp (-1)^{|b|\,|\tilde{b}| \frac{p(p-1)}{2}} \Sigma_{\tilde{b}\ast b}.
\end{equation}
\end{proposition}

\begin{proof}[Sketch of proof]
We introduce a family of Riemann surfaces with $(2p+2)$ marked points, which depends on an additional parameter $\eta \in (1,\infty)$. Each of those surfaces $C_\eta$ is a copy of $C$, and the marked points are $z_{C_\eta,k} = z_{C,k}$, $k \in \{0,\dots,p,\infty\}$, from \eqref{eq:c-curve} together with
\begin{equation} \label{eq:c-curve-2}
\tilde{z}_{C_\eta,1} = \eta z_{C,1},\; \dots,\; \tilde{z}_{C_\eta,p} = \eta z_{C,p}.
\end{equation}
There are natural degenerations at the end of our parameter space: as $\eta \rightarrow 1$,
each point $\tilde{z}_{C_\eta,k}$ collides with its counterpart $z_{C_\eta,k}$, and one can see this as each pair bubbling off into an extra component of a nodal curve $C_1$. As $\eta \rightarrow \infty$, all the $\tilde{z}_{C_\eta,k}$ collide with $z_{C_\eta,\infty}$, and one can see as degeneration of $C_\eta$ into a nodal curve $C_\infty$ with two components, each of which is modelled on the original \eqref{eq:c-curve} (see Figure \ref{fig:2}).
\begin{figure}
\begin{centering}
\includegraphics{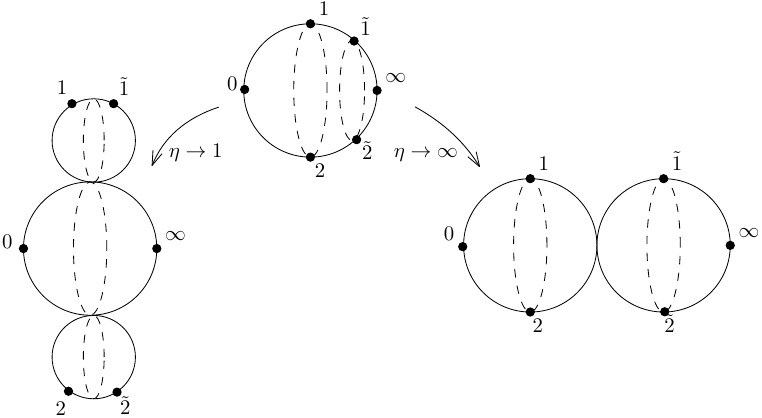}
\caption{\label{fig:2}The family underlying the proof of Proposition \ref{th:sigma-composition}, for $p = 2$.}
\end{centering}
\end{figure}%

We choose an equivariant inhomogeneous term $\nu^{\mathit{eq}}_{C_\eta}$ on each of our curves, which is well-behaved under the two degenerations (and is zero in a neighbourhood of the nodes and marked points; the details are similar to our previous definition of fibrewise inhomogeneous terms). Given critical points $x_0,x_1,\tilde{x}_1,\dots,x_p,\tilde{x}_p,x_\infty$ of the Morse function $f$, and a cell $\Delta_i$, we define a moduli space of triples $(\eta,w,u)$, where: $\eta \in (1,\infty)$, $w$ is as in \eqref{eq:cellrestriction}, and $u: C_\eta \rightarrow M$ is a map, representing the given homology class $A$, which satisfies the $\eta$-parametrized version of \eqref{eq:conditionsformodulispace}, and the incidence conditions \eqref{eq:incidence} as well as
\begin{equation} \label{eq:added-incidence-2}
u(\tilde{z}_{C_\eta,1}) \in W^u(\tilde{x}_1),\; \dots,\; u(\tilde{z}_{C_\eta,p}) \in W^u(\tilde{x}_p).
\end{equation}
To understand the algebraic relations which this parametrized moduli space provides, we have to look at the contributions from limits with $\eta = 1$ or $\eta = \infty$. The $\eta = 1$ contribution is given by a suitable moduli space of maps on $C_1$, and is fairly easy to interpret. Namely, one follows the proof of Lemma \ref{th:p-plus-one-fold} and separates the components of $C_1$ by finite length gradient trajectories (to preserve the $\bZ/p$-symmetry, all the lengths must be the same, so there is only one length parameter). As the length goes to infinity, the Morse flow lines split, and we end up with a composition of quantum product (of $x_k$ and $x_k'$) and a remaining component where we have the previously defined operation \eqref{eq:sigma-i-a}. We can apply the same strategy to the $\eta = \infty$ limit, inserting a finite length gradient flow line between the two pieces. As the length goes to infinity, we end up with two separate components carrying equations of the kind which underlies \eqref{eq:sigma-i-a}. However, the two equations are coupled because they carry the same parameter $w \in S^\infty$. In other words, the resulting moduli spaces end up being
\begin{equation} \label{eq:prod-over}
\bigcup\; \scrM_{A_1}(\Delta_i \times C,x_0,\dots,x_p,x) \times_{S^\infty}
\scrM_{A_2}(\Delta_i \times C,x,\tilde{x}_1,\dots,\tilde{x}_p,x_\infty),
\end{equation}
where the (disjoint) union is over $A_1 + A_2 = A$ and all critical points $x$. 

In the same spirit as in \eqref{eq:sigma-i-a}, we denote the operations obtained from \eqref{eq:prod-over} by 
\begin{equation} \label{eq:sigma-i-a-2}
\Xi_A(\delta(\Delta_i),\dots): \mathit{CM}^*(f) \otimes \mathit{CM}^*(f)^{\otimes 2p} \longrightarrow \mathit{CM}^{*-i-2c_1(A)}(f).
\end{equation}
We also find it convenient to add up over all $A$, with the usual $q^A$ coefficients. Fix cocycles $b$ and $\tilde{b}$ and insert them into \eqref{eq:sigma-i-a-2} at the marked points labeled $(1,\dots,p)$ and $(\tilde{1},\dots,\tilde{p})$, respectively, with signs as in \eqref{eq:multi-b}. This yields a chain map
\begin{equation} \label{eq:multi-b-3}
\Xi_{\tilde{b},b}(\delta(\Delta_i),\cdot):
\mathit{CM}^*(f) \longrightarrow (\mathit{CM}(f) \otimes \Lambda)^{^{*+p|b|+p|\tilde{b}|}}.
\end{equation}
The outcome of the parametrized moduli space argument outlined above is a chain homotopy
\begin{equation} \label{eq:p1}
\Xi_{\tilde{b},b}(\delta(\Delta_i),\cdot) \htp 
\Sigma_{\tilde{b} \ast b}(\Delta_i,\cdot).
\end{equation}
We will be somewhat brief about the final step, since that is a general issue involving equivariant cohomology, and not really specific to our situation. One can construct chain maps like \eqref{eq:multi-b-3} not just for $\delta(\Delta_i)$, but for other $\bF_p$-coefficient cycles in $S^\infty/(\bZ/p) \times S^\infty/(\bZ/p)$, such as $\Delta_{i_1} \times \Delta_{i_2}$. In that case, there is a simple decomposition formula
\begin{equation} \label{eq:p2}
\Xi_{\tilde{b},b}(\Delta_{i_1} \times \Delta_{i_2}, \cdot) =
(-1)^{|b|\, |\tilde{b}|\,\frac{p(p-1)}{2}} \Xi_{\tilde{b}}(\Delta_{i_1}, \Xi_b(\Delta_{i_2},\cdot))
\end{equation}
where the Koszul sign arises from reordering $(\tilde{b},b,\tilde{b},b,\dots)$ into $(\tilde{b},\dots,\tilde{b},b,\dots,b)$. Finally, homologous cycles give homotopic maps. One can use that, and the decomposition of $\delta(\Delta_i)$ into product cycles from Section \ref{subsec:kunneth}, to obtain a further homotopy
\begin{equation} \label{eq:p3}
\Xi_{\tilde{b},b}(\delta(\Delta_i),\cdot) \htp
 \begin{cases} \displaystyle \sum_{i_1+i_2 = i} 
 \Xi_{\tilde{b},b}(\Delta_{i_1} \times \Delta_{i_2},\cdot) & \text{if $i$ is odd or $p=2$,} \\[1em]
\displaystyle \sum_{\substack{i_1+i_2=i \\ \text{$i_k$ even}}} \Sigma_{\tilde{b},b}(\Delta_{i_1} \times \Delta_{i_2},\cdot) & \text{if $i$ is even and $p>2$.}
\end{cases}
\end{equation}
The combination of \eqref{eq:p1}, \eqref{eq:p2} and \eqref{eq:p3} then completes the argument.
\end{proof}

\subsection{\label{subsec:pi-operation}}
We now merge ideas from Sections \ref{subsec:augmented-moduli-spaces-of-holomorphic-curves} and \ref{subsec:equivariant-moduli-space}, by which we mean that we take moduli spaces parametrized by cells in $S^\infty/(\bZ/p)$, and add an additional freely moving marked point to the domain. The starting point is, once more, the family \eqref{eq:s2-family}. 
From its construction as a blowup of $C \times S \rightarrow S$, this inherits a (diagonal) $(\bZ/p)$-action, which we denote by $\sigma_{\scrC}$. 

Fix an almost complex structure $J$. An equivariant fibrewise inhomogeneous term is a complex anti-linear map
\begin{equation}
\nu_{\scrC/S}^{\mathit{eq}}: T(\scrC^{\mathit{reg}}/S) \longrightarrow TM,
\end{equation}
where both bundles have been pulled back to $S^\infty \times_{\bZ/p} \scrC^{\mathit{reg}} \times M$. When restricted to any $S^{2k-1} \times_{\bZ/p} \scrC^{\mathit{reg}} \times M$, it should vanish outside a compact subset (meaning, it's zero in a neighbourhood of $S^{2k-1} \times_{\bZ/p} \scrC^{\mathit{sing}} \times M$; the restriction to $S^{2k-1}$ follows our usual process of treating $S^\infty$ as a direct limit of finite-dimensional manifolds). As before, one can think of it more explicitly as a family $\nu_{\scrC/S,w}^{\mathit{eq}}$ of fibrewise inhomogeneous terms parametrized by $w \in S^\infty$, and satisfying a $(\bZ/p)$-equivariance property as in \eqref{eq:nu-equivariance}: 
\begin{equation}
\nu_{\scrC/S,\tau(w),z,x}^{\mathit{eq}} = \nu_{\scrC/S,w,\sigma_{\scrC}(z),x}^{\mathit{eq}} \circ D\sigma_{\scrC}: T(\scrC^{\mathit{reg}}/S)_z \rightarrow \mathit{TM}_x.
\end{equation}
The associated moduli space consists of triples $(w,v,u)$, where the parameters are $(w,v) \in S^\infty \times S$, $v$ being a regular value of \eqref{eq:s2-family}, and $u:\scrC_v \rightarrow M$ is a solution of the inhomogeneous Cauchy-Riemann equation given by $\nu_{\scrC_v,w}^{\mathit{eq}}$. These inherit a $(\bZ/p)$-action as in \eqref{eq:flip-u}:
\begin{equation}
(w,v,u) \longmapsto (\tau(w),\sigma^{-1}(v),u \circ \sigma_{\scrC}).
\end{equation}
We impose the usual incidence conditions, given by the (un)stable manifolds of critical points $x_0,\dots,x_p,x_\infty$, and by a codimension $2$ submanifold $\Omega$ at the $*$ marked point. Finally, we restrict to the interior of cells \eqref{eq:cellrestriction}. Denote the resulting moduli spaces by $\scrM_A(\Delta_i \times \scrC,x_0,\dots,x_p,x_\infty,\Omega)$. Their expected dimension remains as in \eqref{eq:equivariant-moduli-space-dimension}.

We omit the discussion of transversality and of the compactifications, which is simply a combination of those in Sections \ref{subsec:augmented-moduli-spaces-of-holomorphic-curves} and \ref{subsec:equivariant-moduli-space}. The outcome of isolated-point-counting in our moduli space are maps
\begin{equation} \label{eq:pipi}
\Pi_A(\Delta_i,\dots): \mathit{CM}^*(f) \otimes \mathit{CM}^*(f)^{\otimes p} \longrightarrow \mathit{CM}^{*-i-2c_1(A)}(f)
\end{equation}
which, due to the structure of the compactified one-dimensional moduli spaces, satisfy the same equation as the $\Sigma_A(\Delta_i,\dots)$, see \eqref{eq:parametrized-boundary}. Specializing to coefficients in $\bF_p$, and fixing a Morse cocycle $b$, one can therefore use \eqref{eq:pipi} to define a chain
\begin{equation} \label{eq:multi-b-2}
\Pi_{A,b}: \mathit{CM}^*(f) \longrightarrow (\mathit{CM}(f)[[t,\theta]])^{*+p|b|-2c_1(A)}
\end{equation}
exactly as in \eqref{eq:s-b}. Moreover, up to homotopy that map depends linearly on $[b]$, as in Lemma \ref{th:linearity}. Again up to homotopy, it is also independent of all choices, including that of $\Omega$ within its cohomology class $a = [\Omega] \in H^2(M;\bZ)$. 

\begin{definition} \label{th:q-pi}
For $a \in H^2(M;\bZ)$, $b \in H^*(M;\bF_p)$ and $A \in H_2(M;\bZ)$, we define $Q\Pi_{A,a,b}$ to be the cohomology level map induced by \eqref{eq:multi-b-2}. Adding up those maps with weights $q^A$ yields \eqref{eq:modified-sigma}.
\end{definition}

\begin{proposition} \label{th:pi-divisor}
Fix some $A$ and integer $i$. For suitable choices made in the definition, we have $\Pi_A(\Delta_i,\dots) = (A \cdot \Omega) \Sigma_A(\Delta_i,\dots)$. As a consequence, we have $Q\Pi_{A,a,b} = (A \cdot \Omega) Q\Sigma_{A,b}$ for all $i$ and $A$, which is equivalent to \eqref{eq:pi-divisor}.
\end{proposition}

\begin{proof}
The geometric part of this is exactly as in Proposition \ref{th:divisor-equations-proposition}: for suitably correlated choices of inhomogeneous terms, the underlying moduli spaces bear the same relationship. Since that argument involves making a small perturbation, we can only apply it to finitely many moduli spaces at once, and that explains the bound on $i$ in the statement. As a consequence, we get equality of the $i$-th coefficient in $Q\Pi_{A,a,b}$ and $(A \cdot \Omega) Q\Sigma_{A,b}$. 
\end{proof}

\begin{remark} \label{th:generalized-pi}
Both in Section \ref{subsec:augmented-moduli-spaces-of-holomorphic-curves} and here, we have used an evaluation constraint at a codimension two submanifold $\Omega \subset M$, which limits $Q\Pi_{a,b}$ to $a \in H^2(M;\bZ)$. One can replace that by a pseudo-cycle of arbitrary dimension $d$ (see e.g.\ \cite{zinger10}) and then, the definition goes through without any significant changes for $a \in H^d(M;\bZ)$. In fact, one could even take a mod $p$ pseudo-cycle. This consists of an oriented manifold with boundary $N^d$, such that $\partial N$ carries a free $(\bZ/p)$-action, and a map $f: N \rightarrow M$ such that $f|\partial N$ is $(\bZ/p)$-invariant, with the following properties: the limit points of $f$ are contained in the image of a map from a manifold of dimension $(d-2)$, and the limit points of $f|\partial N$ are contained in the image of a map from a manifold of dimension $(d-3)$. While we do not intend to develop the theory of mod $p$ pseudo-cycles here, this should allow one to define $Q\Pi_{a,b}$ for all $a \in H^d(M;\bF_p)$. The proof of \eqref{eq:pi-relation} given in the next section extends to such generalizations in a straightforward way, but of course, there is no analogue of \eqref{eq:pi-divisor} in codimensions $d>2$.
\end{remark}

\section{Proof of Theorem \ref{th:covariantly-constant}}
This section derives \eqref{eq:pi-relation}. Together with the previously established \eqref{eq:pi-divisor}, that completes our proof of Theorem \ref{th:covariantly-constant}.

\subsection{}
We decompose the moduli spaces underlying $Q\Pi_{a,b}$ into pieces, where the position of the additional marked point is constrained to lie in one of the cells from Section \ref{sec:subsec-basic-top-b}. This means that instead of using $\Delta_i \times S \subset S^\infty \times S$ as parameter spaces, we look at the subspaces $\Delta_i \times W$, where 
\begin{equation}
W \in \{P_0,\, Q_0,\, \sigma^j(L_1), \sigma^j(B_2)\}.
\end{equation}
Within the framework of Section \ref{subsec:pi-operation}, it is unproblematic to ensure that all the resulting moduli spaces, denoted by $\scrM_A(\Delta_i \times \scrC|W,x_0,\dots,x_p,x_\infty, \Omega)$, satisfy the usual regularity and compactness properties. Point-counting in them gives rise to maps
\begin{equation} \label{eq:cell-maps}
\Pi_A(\Delta_i \times W,\dots): \mathit{CM}^*(f) \otimes \mathit{CM}^*(f)^{\otimes p} \longrightarrow \mathit{CM}^{*-i-2c_A(A)-|W|+2}(f).
\end{equation}
As in \eqref{eq:parametrized-boundary}, adjacencies between cells determine relations between the associated invariants. In our case, these are governed by \eqref{eq:partial-delta-1}--\eqref{eq:partial-delta-2} and \eqref{eq:diff-s2-1}--\eqref{eq:diff-s2-3}. Explicitly, the relations are
\begin{equation}
\begin{aligned}
& d\Pi_A(\Delta_i \times W, x_0,\dots,x_p) - (-1)^{i+|W|} \sum_{k=0}^p (-1)^{|x_0|+\cdots+|x_{k-1}|} \Pi_A(\Delta_i \times W,x_0,\dots,dx_k,\dots,x_p) 
\\ &
= \begin{cases} \displaystyle
\sum_j (-1)^* \Pi_A(\Delta_{i-1} \times \sigma^j W,x_0,x_1^{(j)},\dots,x_p^{(j)}) & \text{$i$ even,} \\ \displaystyle
(-1)^* \Pi_A(\Delta_{i-1} \times \sigma W,x_0,x_1^{(1)},\dots,x_p^{(1)}) - \Pi_A(\Delta_{i-1} \times W,x_0,x_1,\dots,x_p) & \text{$i$ odd}
\end{cases}
\\[1em] &
+ \text{(extra term depending on $W$)}.
\end{aligned}
\end{equation}
The last-mentioned term is zero if $W \in \{P_0,Q_0\}$, with the remaining cases being
\begin{align}
& 
\begin{aligned} &
(\text{extra term for $W = \sigma^jL_1$}) \\ & = (-1)^i \big(
\Pi_A(\Delta_i \times Q_0,x_0,x_1^{(j)},\dots,x_p^{(j)}) - \Pi_A(\Delta_i \times P_0,x_0,x_1^{(j)},\dots,x_p^{(j)}) \big), 
\end{aligned}
\\
&
\begin{aligned} &
(\text{extra term for $W = \sigma^jB_2$})
\\ & = (-1)^{i+1} \big(
\Pi_A(\Delta_i \times \sigma^{j+1} L_1,x_0,x_1^{(1)},\dots,x_p^{(1)}) - \Pi_A(\Delta_i \times \sigma^j L_1,x_0,x_1,\dots,x_p) \big).
\end{aligned}
\end{align}
As usual, we now specialize to coefficients in $\bF = \bF_p$. The relations above immediately imply the following:

\begin{lemma}
Fix a cocycle $b \in \mathit{CM}^*(f)$. Then, the $t$-linear map
\begin{equation} \label{eq:fixed-b-map}
\begin{aligned}
& \Pi_{A,b}^{\mathit{eq}}: C_{-*}(S)_{\mathit{eq}} \otimes \mathit{CM}^*(f) \longrightarrow (\mathit{CM}(f)[[t,\theta]])^{*+p|b|-2c_1(A)+2}, \\
&
W \otimes x \longmapsto (-1)^{|b|\,(|W|+|x|)} \sum_k \Big( \Pi_A(\Delta_{2k} \times W,x,b,\dots,b)
\\
&
\qquad \qquad \qquad \qquad +
(-1)^{|x|+|b|+|W|} \Pi_A(\Delta_{2k+1} \times W,x,b,\dots,b) \theta \Big)t^k, 
\\
&
W\,\theta \otimes x \longmapsto (-1)^{|b|\,(|W|+|x|)}
 \sum_k \Big( (-1)^{|x|} \Pi_A(\Delta_{2k} \times W,x,b,\dots,b) \theta \\
& \qquad \qquad \qquad \qquad - (-1)^{|b|+|W|}
 \sum_j  j \Pi_A(\Delta_{2k+1} \times \sigma^j W,x,b,\dots,b) t \Big)t^k,
\end{aligned}
\end{equation}
is a chain map.
\end{lemma}

Following \eqref{eq:equivariant-complex-map}, one can think of \eqref{eq:fixed-b-map} as a special case of a more general structure, which would be a $t$-linear chain map
\begin{equation} \label{eq:very-general}
(C_{-*}(S) \otimes \mathit{CM}^*(f) \otimes \mathit{CM}^*(f)^{\otimes p})_{\mathit{eq}} \longrightarrow (\mathit{CM}(f)[[t,\theta]])^{*-2c_1(A)+2}.
\end{equation}
Here, the group $\bZ/p$ acts on $C_{-*}(S)$, as well as on $\mathit{CM}^*(f)^{\otimes p}$ by cyclic permutations. As in the previous situation, \eqref{eq:very-general} would be useful in order to prove that \eqref{eq:fixed-b-map} only depends on the cohomology class of $b$, and is additive. For our purposes, however, we can work around that, since all necessary computations can be done using a fixed cocycle $b$.

\subsection{}
At this point, everything we need can be extracted from an analysis of the chain map \eqref{eq:fixed-b-map}.
\begin{lemma} \label{th:specialize-1}
Suppose that we specialize \eqref{eq:fixed-b-map} to using $W=B_2 + \sigma B_2 + \cdots + \sigma^{p-1}B_2 \in C_2(S)_{\mathit{eq}}$. Then, the resulting chain map $\mathit{CM}^*(f) \rightarrow (\mathit{CM}(f)[[t,\theta]])^{*+p|b|-2c_1(A)}$ is equal to $\Pi_{A,b}$.
\end{lemma} 

\begin{proof}
This is essentially by definition. We are considering the map
\begin{equation} \label{eq:spec}
\begin{aligned} 
x \longmapsto (-1)^{|b|\,|x|} & \sum_{j,k} \Big( \Pi_A(\Delta_{2k} \times \sigma^j(B_2),x,b,\dots,b) \\[-1em] & \qquad \qquad +
(-1)^{|b|+|x|} \Pi_A(\Delta_{2k+1} \times \sigma^j(B_2),x,b,\dots,b) \theta \Big) t^k.
\end{aligned}
\end{equation}
The regularity of the spaces $\scrM_A(\Delta \times \scrC|W, x_0,\dots,x_p,x_\infty,\Omega)$ for cells $W$ of dimension $<2$ implies that in a zero-dimensional space $\scrM_A(\Delta \times \scrC,x_0,\dots,x_p,x_\infty,\Omega)$, none of the points arises from a parameter value $v \in S$ which belongs to one of those cells. In other words, that space $\scrM_A(\Delta \times \scrC,x_0,\dots,x_p,x_\infty,\Omega)$ is the disjoint union of $\scrM_A(\Delta \times \scrC|\sigma^j B,x_0,\dots,x_p,x_\infty,\Omega)$. 
\end{proof}

\begin{lemma} \label{th:specialize-2}
Suppose that we specialize \eqref{eq:fixed-b-map} to using $W=P_0 \in C_0(S)_{\mathit{eq}}$, and pass to cohomology. Then, the resulting map is equal to the following: take all possible decompositions $A = A_1+A_2$, and add up 
\begin{equation} \label{eq:comp-1}
H^*(M;\bF_p) \xrightarrow{\ast_{A_1} a}
H^{*+2-2c_1(A_1)}(M;\bF_p) \xrightarrow{Q\Sigma_{b,A_2}} (H(M;\bF_p)[[t,\theta]])^{*+p|b|+2-2c_1(A)},
\end{equation}
where $a = [\Omega] \in H^2(M;\bZ)$.
\end{lemma} 

\begin{proof}
This time, the reason is geometric. Using $P_0$ means that we are restricting to a particular fibre of \eqref{eq:s2-family}, which is the nodal surface from Figure \ref{fig:nodal}(i). Recall that each component of that surface carries an inhomogeneous term,  which additionally depends on parameters in $S^\infty$. However, without violating regularity or other restrictions, one can arrange that the inhomogeneous term on the component which is a three-pointed sphere ($\scrC_{0,-}$ in the notation from Section \ref{subsec:augmented-moduli-spaces-of-holomorphic-curves}) is independent of those parameters.
\begin{figure}
\begin{centering}
\includegraphics{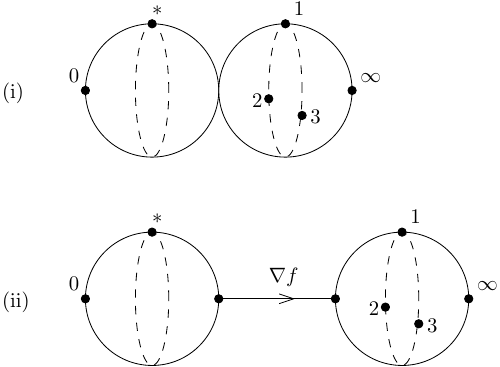}
\caption{\label{fig:nodal}A schematic picture of the proof of Lemma \ref{th:specialize-2}, with $p = 3$.}
\end{centering}
\end{figure}
After that, one inserts a finite length Morse flow line between the two components, as in Figure \ref{fig:nodal}(ii). In the same way as in Lemma \ref{th:p-plus-one-fold}, the resulting (varying length) moduli space gives a chain homotopy between our operation and the chain map underlying the composition \eqref{eq:comp-1}, in its Morse-theoretic incarnation.
%
\end{proof}

\begin{lemma} \label{th:specialize-3}
Suppose that we specialize \eqref{eq:fixed-b-map} to using $W=Q_0 \in C_0(S)_{\mathit{eq}}$, and pass to cohomology Then, the resulting map is equal to the following: take all possible decompositions $A = A_1+A_2$, and add up 
\begin{equation} \label{eq:comp-2}
H^*(M;\bF_p) \xrightarrow{Q\Sigma_{b,A_1}}
H(M;\bF_p)[[t,\theta]])^{*+p|b|-2c_1(A_1)} \xrightarrow{\ast_{A_2} a}
(H(M;\bF_p)[[t,\theta]])^{*+p|b|+2-2c_1(A)},
\end{equation}
where $a = [\Omega]$ as before.
\end{lemma}

The proof is the same as for Lemma \ref{th:specialize-2}. Note that the operations in \eqref{eq:comp-2} appear in the opposite order from \eqref{eq:comp-1}. The reason is that over $v = 0$, the component $\scrC_{y,-}$ is attached to $\scrC_{v,+}$ at the point $0 \in C$, which serves as input of the $\Sigma$ operation; whereas for $v =\infty$, it is attached at the output point $\infty \in C$. Finally, we have the following, which establishes \eqref{eq:pi-relation}:

\begin{proposition}
$t Q\Pi_{a,b}$ equals the difference between \eqref{eq:comp-1} and \eqref{eq:comp-2}.
\end{proposition}

\begin{proof}
By Lemma \ref{th:specialize-1}, $\Pi_{A,b} t$ is obtained by specializing \eqref{eq:fixed-b-map} to $(B_2+\cdots+\sigma^{p-1}B_2)t$. From \eqref{eq:coh-localize-1} and \eqref{eq:coh-localize-2}, we see that this is chain homotopic to specializing the same map to $(P_0-Q_0)$. Using Lemma \ref{th:specialize-2} and \ref{th:specialize-3} then yields the desired result.
\end{proof}

\section{Computations}
\label{sec:computations}
In this section, we explore the power of Theorem \ref{th:covariantly-constant} as a computational tool.

\subsection{\label{subsec:s2}}
Our first task is to work out the details of Example \ref{th:s2}, where $M$ is the two-sphere. We use the standard generator of $H_2(M;\bZ)$, and correspondingly write $\Lambda$ as a power series ring in one variable $q$. The quantum connection is
\begin{equation} \label{eq:explicit-nabla-s2}
\nabla = t q\partial_q + \begin{pmatrix} 0 & q \\ 1 & 0 \end{pmatrix}. 
\end{equation}
Let's temporarily use $\bQ$-coefficients, and allow inverses of $t$. If $\xi$ satisfies
\begin{equation} \label{eq:2nd-order}
(tq \partial_q)^2 \xi = q\xi,
\end{equation}
then the following endomorphism is covariantly constant with respect to \eqref{eq:explicit-nabla-s2}:
\begin{equation} \label{eq:explicit-xi}
\Xi = \begin{pmatrix}
-\xi (tq\partial_q \xi) &
-(t q\partial_q \xi)^2 \\
\xi^2 &
\xi (tq\partial_q \xi) 
\end{pmatrix}.
\end{equation}
It is straightforward to write down an explicit solution of \eqref{eq:2nd-order}:
\begin{equation} \label{eq:explicit-xi-1}
\xi = \sum_{k=0}^{\infty} \frac{1}{(k!)^2} q^k t^{-2k}.
\end{equation}

Pick a prime $p>2$. Take \eqref{eq:explicit-xi} with \eqref{eq:explicit-xi-1}, and truncate it by dropping all powers $q^p$ or higher. The remaining denominators are coprime to $p$, so we can reduce coefficients to $\bF_p$. The outcome, using some elementary combinatorics to simplify the formulae, is the matrix $\Sigma$ from \eqref{eq:explicit-xi-2}. For example, the $q^k t^{-2k}$ term of the $\sigma_{21}$ coordinate of \eqref{eq:explicit-xi} is $$\sum_{k=k_1+k_2} \tfrac{1}{(k_1!)^2} \tfrac{1}{(k_2!)^2} = \frac{1}{(k!)^2} \sum_{k=k_1+k_2} {{k}\choose{k_1}}{{k}\choose{k_2}} = \frac{1}{(k!)^2} {{2k}\choose{k}},$$ the second equality being the Chu-Vandermonde identity. We notice that this is the $\sigma_{21}$ component $\tfrac{(2k)!}{(k!)^4}$ of \eqref{eq:explicit-xi-2}. Similarly, the coefficient of $q^k t^{2-2k}$ in the $\sigma_{12}$ component of \eqref{eq:explicit-xi} is $$ \sum_{k=k_1+k_2} \tfrac{k_1}{(k_1!)^2} \tfrac{k_2}{(k_2!)^2},$$ and by using the Chu-Vandermonde identity on $$\tfrac{1}{(k-1)!^2} \sum_{k=k_1+k_2}{{k-1}\choose{k_1}} {{k-1}\choose{k_2}},$$ one obtains the coefficient of $q^k t^{2-2k}$ in the $\sigma_{12}$ component of \eqref{eq:explicit-xi-2}. A similar application of this identity can be used for the $\sigma_{11}$ component.

By construction, this endomorphism is covariantly constant modulo $q^p$; and the constant term (in $q$) of $-t^{p-1}\Sigma$ matches the cup product with $\mathit{St}(h) = -t^{p-1}h$ (see \eqref{eq:trivial-steenrod-2} for the sign convention). Therefore, $-t^{p-1}\Sigma$ and $Q\Sigma_h$ must agree modulo $q^p$. But for degree reasons, $Q\Sigma_h$ can't have terms of order $q^p$ or higher. The consequence is that $Q\Sigma_h = -t^{p-1}\Sigma$, as previously stated.

\begin{remark}
It is worthwhile spelling out the comparison with the fundamental solution of the quantum differential equation, mentioned in Remark \ref{th:fundamental-solution}. For $S^2$, the fundamental solution is \cite[Section 28.2]{clay} (note the differences in notation and conventions: our $t$ is their $-\hbar$; our $q$ is their $e^t$; our $t$ is their $H$)
\begin{equation}
\Psi = \begin{pmatrix} -tq\partial_q \eta & -t q\partial_q \xi \\ \eta & \xi \end{pmatrix}, \end{equation}
where $\xi$ is as in \eqref{eq:explicit-xi-1}, and
\begin{equation}
\eta = \sum_{k=0}^\infty \frac{1}{(k!)^2} q^k t^{-2k-1} \big(-\log(q) + 2 \sum_{j=1}^k \frac{1}{j} \big) 
\end{equation}
is a multivalued solution of the same equation \eqref{eq:2nd-order} as $\xi$. By forming \eqref{eq:xi-a} with $\beta = h$, one gets exactly the matrix from \eqref{eq:explicit-xi}:
\begin{equation} \label{eq:new-xi}
\Xi = \Psi\begin{pmatrix} 0 & 0 \\ 1 & 0 \end{pmatrix} \Psi^{-1}.
\end{equation}
\end{remark}

%

\subsection{}
Following ideas from \cite{wilkins18}, let's look at the following situation:

\begin{assumption} \label{th:2-generates}
The second cohomology group $H^2(M;\bF_p)$ generates $H^*(M;\Lambda)$ as a ring, with the quantum product.
\end{assumption}

This implies that $H^*(M;\bF_p)$ is zero in odd degrees. It also implies that each class in $H^2(M;\bF_p)$ can be lifted to $H^2(M;\bZ)$, as one sees by looking at
\begin{equation}
\cdots \rightarrow H^2(M;\bZ) \rightarrow H^2(M;\bF_p) \rightarrow H^3(M;\bZ) \xrightarrow{p} H^3(M;\bZ) \rightarrow H^3(M;\bF_p) \rightarrow \cdots
\end{equation}
%

\begin{lemma} \label{th:2-generates-2}
Suppose that Assumption \ref{th:2-generates} holds. Then, the quantum product and $\mathit{QSt}(b)$, for $b \in H^2(M;\bF_p)$, determine all the quantum Steenrod operations.
\end{lemma}

\begin{proof}
Write the covariant constancy property as
\begin{equation} \label{eq:rewrite-constancy}
\mathit{Q\Sigma}_b(a \ast c) = t \partial_a \mathit{Q\Sigma}_b(c) + a \ast \mathit{Q\Sigma}_b(c),
\;\; a \in H^2(M;\bZ), \;\; b,c \in H^*(M;\bF_p).
\end{equation}
This shows that $Q\Sigma_b(c)$ and the quantum product determine $Q\Sigma_b(a \ast c)$. 
Therefore, if one knows $\mathit{QSt}(b) = Q\Sigma_b(1)$ and Assumption \ref{th:2-generates-2} holds, the entire operation $Q\Sigma_b$ can be computed from that. By \eqref{eq:compose-sigma},
\begin{equation} \label{eq:so-called-cartan}
\mathit{QSt}(b \ast c) = Q\Sigma_{b}(Q\mathit{St}(c)).
\end{equation}
If we know $\mathit{QSt}(b)$ and $\mathit{QSt}(c)$, for some $b \in H^2(M;\bF_p)$ and $c \in H^*(M;\bF_p)$, then our previous argument determines $Q\Sigma_{b}$, and we can get $\mathit{QSt}(b \ast c)$ from that by \eqref{eq:so-called-cartan}. In view of Assumption \ref{th:2-generates-2}, this implies the desired result.
\end{proof}

Here is a concrete class of examples to which this strategy applies.

\begin{proposition} \label{th:fano}
Suppose that $M$ is a monotone symplectic manifold, satisfying Assumption \ref{th:2-generates}. Then the quantum Steenrod operations can be computed in terms of the quantum product and classical Steenrod operations.
\end{proposition}

\begin{proof}
Take $b \in H^2(M;\bF_p)$. Then $\mathit{QSt}(b)$ has degree $2p$. The monomials in it that can have nonzero coefficients are $t^j q^A$, where $j+c_1(A) \leq p$. The terms with $j = 0$ and $c_1(A) = p$ are part of \eqref{eq:non-equivariant}. The remaining terms are determined by covariant constancy, since any monomial $q^A$ that lies in $I_{\mathit{diff}}$ must necessarily have $c_1(A) \geq p$. Having determined $\mathit{QSt}(b)$, Lemma \ref{th:2-generates-2} does the rest.
\end{proof}

As a concrete illustration, let's consider a cubic surface $M \subset \bC P^3$, which is a del Pezzo surface, and hence a monotone symplectic manifold. For simplicity, instead of the whole Novikov ring, we will work with a single Novikov variable $q$, which counts the Chern number of holomorphic curves. Let's first take coefficients in $\bZ$. Take $h_2$ to be the first Chern class of $M$, and $h_4$ to be the Poincar{\'e} dual of a point. Computations in \cite{crauder-miranda94, difrancesco-itzykson94} show that
\begin{equation} \label{eq:cubic-quantum}
\begin{aligned}
& h_2 \ast h_2 = 3h_4 + 9q\, h_2 + 108q^2, \\
& h_2 \ast h_4 = 36 q^2\, h_2 + 252 q^3.
\end{aligned}
\end{equation}
At one point we will use another class in $H^2(M)$, the Poincar{\'e} dual of a Lagrangian sphere, denoted by $l_2$. This satisfies
\begin{equation} \label{eq:l2-square}
l_2 \ast l_2 = -2h_4 + 4q\, h_2 + 12 q^2.
\end{equation}

\begin{example}
Take the cubic surface with $p = 2$ (this computation is of the same kind as those in \cite{wilkins18}, only expressed in slightly different language). First of all,
\begin{equation} \label{eq:st-degree-2}
\mathit{QSt}(c) = c \ast c + tc \;\; \text{for all $c \in H^2(M;\bF_2)$.}
\end{equation}
A priori, $\mathit{QSt}(c)$ could also have a $tq$ term, which would lie in $H^0(M;\bF_2)$. This would come from classes with $c_1(A) = 1$. To get a nonzero output in $H^0(M;\bF_2)$, one would need to have a stable $A$-curve going through every point of $M$. But each $A$ is represented by a unique embedded $(-1)$-sphere, hence the term must vanish, leaving \eqref{eq:st-degree-2}. 

By combining \eqref{eq:cubic-quantum}, \eqref{eq:st-degree-2}, and \eqref{eq:rewrite-constancy}, one gets
\begin{equation}
\begin{aligned}
& \mathit{QSt}(h_2) = h_4 + (q+t)h_2, \\
& Q\Sigma_{h_2}(h_2) = tq\partial_q \mathit{QSt}(h_2) + h_2 \ast \mathit{QSt}(h_2) = (q+t)h_4 + q^2 h_2, \\
& Q\Sigma_{h_2}(h_2 \ast h_2) = tq\partial_q Q\Sigma_{h_2}(h_2) + h_2 \ast Q\Sigma_{h_2}(h_2) 
= (tq+q^2) h_4 + q^3 h_2. 
\end{aligned}
\end{equation}
Using \eqref{eq:so-called-cartan}, we get the result announced in Example \ref{th:cubic-surface-mod-2}:
\begin{equation}
\begin{aligned}
& \mathit{QSt}(h_4) = \mathit{QSt}(h_2 \ast h_2 + qh_2) = 
Q\Sigma_{h_2}(Q\mathit{St}(h_2)) + q^2 \mathit{QSt}(h_2) 
\\ & \qquad = Q\Sigma_{h_2}(h_2 \ast h_2 + th_2) + q^2 \mathit{QSt}(h_2)
= t^2 h_4.
\end{aligned}
\end{equation}
\end{example}

\begin{example}
Let's again look at the cubic surface, but now with $p = 3$. Here, the fact that we work with a single Novikov variable $q$ will limit the effectiveness of our computation, leading to an incomplete result.
As explained in Proposition \ref{th:fano}, we can use covariant constancy to determine the quantum Steenrod operations on $H^2(M;\bF_3)$. In the same way, one can compute $Q\Sigma_{b}(c)$ for $b,c \in H^2(M;\bF_3)$ except for the $q^3 t$ term, which lies in $H^0(M;\bF_3)$. We will only describe the outcome (code that carries out this computation is available at \cite{seidel20}):
\begin{equation}
\begin{aligned}
& \mathit{QSt}(h_2) = -t^2 h_2, \\
& \mathit{QSt}(l_2) = -t^2 l_2, \\
& \mathit{Q\Sigma}_{l_2}(l_2) = -t^2 h_4 + (\text{\it term lying in } H^0(M;\bF_3) q^3t).
\end{aligned}
\end{equation}
From that one gets, using \eqref{eq:l2-square},
\begin{equation}
\begin{aligned}
& \mathit{QSt}(h_4) = \mathit{QSt}(l_2 \ast l_2 - q h_2) = Q\Sigma_{l_2}(\mathit{QSt}(l_2)) - q^3 \mathit{QSt}(h_2) \\ & \qquad = t^4 h_4 + q^3t^2 h_2 + (\text{\it term lying in } H^0(M;\bF_3) q^3t^3).
\end{aligned}
\end{equation}
Note that, unlike the $p = 2$ case, $\mathit{QSt}(h_4)$ contains a non-classical (quantum) term.
\end{example}

\subsection{}
We conclude our discussion with a higher-dimensional case: the intersection of two quadrics in $\bC P^5$, which is a monotone symplectic $6$-manifold. Let's first work with $\bZ$-coefficients. The even degree cohomology has a basis $\{1,h_2,h_4,h_6\}$, where the subscript denotes the dimension. There is also odd degree cohomology, $H^3(M;\bZ) = \bZ^4$, but that will play no role in our argument. We can identify the Novikov ring with $\bZ[[q]]$, but since $c_1(M)$ is twice the positive area generator of $H^2(M;\bZ)$, the formal variable $q$ has degree $4$. The quantum product, as computed in \cite{donaldson92}, satisfies
\begin{equation} \label{eq:donaldson}
\begin{aligned}
& h_2 \ast h_2 = 4(h_4 + q), \\
& h_2 \ast h_4 = h_6 + 2 q h_2, \\
& h_2 \ast h_6 = 4q h_4 + 4q^2, \\
& h_4 \ast h_4 = 2q h_4 + 3q^2.
\end{aligned}
\end{equation}

\begin{example}
Taking our intersection of quadrics, let's set $p = 2$. The classical Steenrod operations are
\begin{equation} \label{eq:classical-steenrod-3-fold}
\mathit{Sq}(h_k) = t^{k/2} h_k.
\end{equation}
For $h_2$, this is because $h_2^2 = 0 \in H^4(M;\bF_2)$, which one can read off from the classical term in \eqref{eq:donaldson}. For $h_4$, its Poincar{\'e} dual of is represented by a line $\bC P^1 \subset M$. The normal bundle of that line has first Chern class $0$; by the geometric description of Steenrod squares through Stiefel-Whitney classes, this implies vanishing of $\mathit{Sq}^2(h_4)$.

Since the quantum product with $h_2$ agrees with its classical counterpart, the cup product with any element of $H^*(M;\bF_2)$ is a covariantly constant endomorphism for the quantum connection. From that, Theorem \ref{th:covariantly-constant}, and \eqref{eq:classical-steenrod-3-fold}, one gets
\begin{equation}
\begin{aligned}
& Q\Sigma_{h_2}(c) = th_2 c + \text{(terms lying in $H^k(M;\bF_2)$ with $k < |c|-4$),} \\
& Q\Sigma_{h_4}(c) = t^2 h_4 c + q^2 c + \text{(terms lying in $H^k(M;\bF_2)$ with $k < |c|$).}
\end{aligned}
\end{equation}
Therefore,
\begin{equation}
\begin{aligned}
& \mathit{QSt}(h_2) = th_2, \\
& \mathit{QSt}(h_4) = t^2 h_4 + q^2, \\
& \mathit{QSt}(h_6) = Q\Sigma_{h_2 \ast h_4}(1) = Q\Sigma_{h_2}(Q\mathit{St}(h_4)) 
= Q\Sigma_{h_2}(t^2h_4 + q^2) = t^3h_6 + q^2t h_2.
\end{aligned}
\end{equation}
\end{example}

\begin{example}
Still for our intersection of quadrics, take $p = 3$. Then, the quantum product and covariant constancy completely determine $Q\Sigma_{h_2}$, for degree reasons (in fact, the same is true for any $p>2$). Explicitly (see again \cite{seidel20} for code), the action on $H^{\mathit{even}}(M;\bF_3)$ is
\begin{equation}
Q\Sigma_{h_2} = \begin{pmatrix} 
qt & q^2 & -q^2t & q^3 \\
-t^2 & qt & 0 & q^2t \\
0 & -t^2 +q & -qt & q^2 \\
1 & 0 & -t^2 & -qt
\end{pmatrix}.
\end{equation}
From that, we get
\begin{equation}
\begin{aligned}
&
Q\mathit{St}(h_2) = Q\Sigma_{h_2}(1) = qt \, 1 - t^2 \, h_2 + h_6, \\
&
Q\mathit{St}(h_4) = Q\Sigma_{h_4}(1) = Q\Sigma_{h_2*h_2 - q1}(1) = Q\Sigma_{h_2}^2(1) - q^3\,1 \\
\notag &
\qquad =  qt(q + t^2)\, h_2 + (q + t^2)^2\, h_4, \\
&
Q\mathit{St}(h_6) = Q\Sigma_{h_6}(1) = Q\Sigma_{h_2*h_2*h_2}(1) = Q\Sigma_{h_2}^3(1) \\
\notag & 
\qquad = q^2 t (q^2 - q t^2 - t^4)\, 1 + q^2 t^4\, h_2 + qt^3 (q + t^2)\, h_4 + (q^3 - q^2 t^2 + q t^4 - t^6) h_6.
\end{aligned}
\end{equation}
\end{example}

\end{document}